\documentclass[11pt,oneside,reqno]{amsart}
\usepackage{amsmath, amssymb, amsthm}
\usepackage{url}
\usepackage{mathrsfs}
\usepackage[ansinew]{inputenc}
\usepackage[breaklinks]{hyperref}
\usepackage{comment}
\usepackage{multirow}

\allowdisplaybreaks

\renewcommand{\Re}{\operatorname{Re}}

\newcommand{\s}{{\sigma}}

 \renewcommand{\a}{\alpha}
\renewcommand{\b}{\beta}

\renewcommand{\l}{\lambda}

\renewcommand{\(}{\left\(}
\renewcommand{\)}{\right\)}
\renewcommand{\[}{\left\[}
\renewcommand{\]}{\right\]}
\newtheorem{remark}[]{Remark}
\numberwithin{equation}{section}
 \theoremstyle{plain}
\newtheorem{theorem}{Theorem}[section]

\newcommand{\sm}{\left(\begin{smallmatrix}}
\newcommand{\esm}{\end{smallmatrix}\right)}

\newtheorem{corollary}[theorem]{Corollary}
\newtheorem{proposition}[theorem]{Proposition}

   \makeatletter
\def\proof{\@ifnextchar[{\@oproof}{\@nproof}}
\def\@oproof[#1][#2]{\trivlist\item[\hskip\labelsep\textit{#2 \textbf{Proof of}\
#1.}~]\ignorespaces}
\def\@nproof{\trivlist\item[\hskip\labelsep\textit{Proof.}~]\ignorespaces}

\makeatother

\usepackage{color}
\usepackage{amsmath}

\definecolor{blue}{rgb}{0,0,1}
\definecolor{red}{rgb}{1,0,0}
\definecolor{green}{rgb}{0,.6,.2}
\definecolor{purple}{rgb}{1,0,1}

\long\def\red#1\endred{{\color{red}#1}}
\long\def\blue#1\endblue{{\color{blue}#1}}
\long\def\purple#1\endpurple{{\color{purple}#1}}
\long\def\green#1\endgreen{{\color{green}#1}}

\usepackage[
    left=1in,
    right=1in,
    top=1in,
    bottom=1in,
    headheight=14pt,
    headsep=0.25in,
    footskip=0.5in
]{geometry}

\begin{document}
\title[On special values of Koshliakov zeta functions]{On special values of Koshliakov zeta functions}
\author{Yashovardhan Singh Gautam}
\address{Department of Mathematics, Indian Institute of Technology, Roorkee-247667, Uttarakhand, India}
\email{yashovardhansg@ma.iitr.ac.in}
\author{Rahul Kumar}
\address{Department of Mathematics, Indian Institute of Technology, Roorkee-247667, Uttarakhand, India}
\email{rahul.kumar@ma.iitr.ac.in} 


\maketitle
\pagenumbering{arabic}
\pagestyle{headings}
\begin{abstract}
	In this paper, we study the Koshliakov zeta function $\eta_p(s)$, whose theory appears to be more involved than that of its counterpart $\zeta_p(s)$, owing to the fact that its defining series is not of Dirichlet type. We derive formulas for $\eta_p(s)$ at both even and odd values of $s$. In the limiting case $p\to\infty$, our results yield the celebrated formulas of Euler and Ramanujan for the Riemann zeta function. Moreover, our results lead to several consequences concerning closed-form expressions for Lambert series and their arithmetic properties, recovering results due to Berndt, Cauchy, Ramanujan, and others. We also propose $p$-analogues of the transformation formula for the classical Eisenstein series. Moreover, we introduce two families of $p$-analogues of Ramanujan polynomials and establish functional equations satisfied by them. 
\end{abstract}

\maketitle

\tableofcontents
\section{Introduction and Motivation}

The theory of the Koshliakov zeta functions, denoted by $\zeta_p(s)$ and $\eta_p(s)$, is an active and evolving area of research. These functions were originally introduced and studied extensively by Koshliakov \cite{koshliakov} in the 1940s and have recently been revived and further studied by Dixit and Gupta \cite{dixitgupta1}. These zeta functions are defined as follows \cite[Equations (2.5)--(2.8)]{dixitgupta1}:
\begin{align}\label{zetap defn}
\zeta_p(s):=\sum_{j=1}^\infty\frac{p^2+\lambda_j^2}{p\left(p+\frac{1}{\pi}\right)+\lambda_j^2}\frac{1}{\lambda_j^s}\qquad \textup{(Re}(s)>1)
\end{align}
and
\begin{align}\label{etap defn}
\eta_p(s):=\sum_{k=1}^\infty\frac{(s,2\pi pk)_k}{k^s} \qquad \textup{(Re}(s)>1),
\end{align}
where $\lambda_1,\lambda_2,\cdots$ are the positive roots  of the equation
\begin{align}\label{sincos eqn}
p\sin(\pi\lambda)+\lambda\cos(\pi\lambda)=0,\ p>0;
\end{align} 
and the function $(s,\nu k)_k$ is defined as
\begin{align}\label{gamma generalized}
(s,\nu k)_k:=\frac{1}{\Gamma(s)}\int_0^\infty e^{-x}\left(\frac{k\nu-x}{k\nu+x}\right)^kx^{s-1}dx.
\end{align}
The origin of these zeta functions lies in a problem from physics, for further details, we refer the reader to \cite{dixitgupta1} and \cite{koshliakov}.

Note that, as $p\to\infty$, both $\zeta_p(s)$ and $\eta_p(s)$ reduce to the Riemann zeta function $\zeta(s)$:
\begin{align*}
\lim_{p\to\infty}\zeta_p(s)=\zeta(s),\qquad \mathrm{and}\qquad \lim_{p\to\infty}\eta_p(s)=\zeta(s).
\end{align*}
This follows from the fact that, as $p\to\infty$, the roots of the equation  \eqref{sincos eqn}, or equivalently of
\begin{align}
 \sin(\pi\lambda)+\frac{\lambda}{p}\cos(\pi\lambda)=0,\nonumber
\end{align}
are precisely the positive integers, that is, $\lambda_j\to j$. Moreover, note that
\begin{align}\label{bracket}
\lim_{p\to\infty}(s,2\pi p k)_k=1.
\end{align} 

Koshliakov demonstrated that $\zeta_p(s)$ and $\eta_p(s)$ possess a number of striking properties, closely analogous to those of the Riemann zeta function $\zeta(s)$. For instance, the Koshliakov zeta functions satisfy the following elegant functional equation \cite[Chapter 1, Equation (30)]{koshliakov}
\begin{align}\label{kosh fe}
\zeta_p(1-s)=2^{1-s}\pi^{-s}\cos\left(\frac{\pi s}{2}\right)\Gamma(s)\eta_p(s),
\end{align}
which is a generalization of the well-known functional equation of the Riemann zeta function 
\begin{align}\label{riemann fe}
\zeta(1-s)=2^{1-s}\pi^{-s}\cos\left(\frac{\pi s}{2}\right)\Gamma(s)\zeta(s).
\end{align}
Moreover, Koshliakov showed that  Euler's classical formula \cite[p.~266]{apostal} for $\zeta(2m), m\in\mathbb{N}$:
\begin{align}\label{zeta(2m)}
\zeta(2m)=\frac{(-1)^{m+1}(2\pi)^{2m}}{2(2m)!}B_{2m},
\end{align}
can be generalized in the setting of $\zeta_p(2m)$ in the following form \cite[Chapter 1, Equation (38)]{koshliakov}:
\begin{align}\label{zeta_p(2m)}
\zeta_p(2m)=\frac{(-1)^{m+1}(2\pi)^{2m}}{2(2m)!}B_{2m}^{(1,p)}.
\end{align}
Here $B_{2m}$ are the classical Bernoulli numbers, and $B_{2m}^{(1,p)}$ denote the first kind\footnote{Koshliakov \cite{koshliakov}, as well as Dixit and Gupta \cite{dixitgupta1}, used the notation $B_{2m}^{(p)}$ for the generalized Bernoulli numbers. however, we use $B_{2m}^{(1,p)}$ since we will also define a second kind of Bernoulli numbers in this paper (see \eqref{B2m2p}).} of $p$-generalization of  Bernoulli numbers, and defined by \cite[Chapter 2, Equation (45)]{koshliakov}
\begin{align}\label{bernoulli p1}
B_{2m}^{(1,p)}:=(-1)^{k+1}4m\int_0^\infty x^{2m-1}\sigma_p(2\pi x)dx, \quad B_0^{(1,p)}:=\frac{1}{1+\frac{1}{\pi p}},
\end{align}
where $\sigma_p(t)$ is given by
\begin{align}\label{sigma_p}
\sigma_p(t):=\sum_{j=1}^{\infty}\frac{p^2+\l_j^2}{p\left(p+\frac{1}{\pi}\right)+\l_j^2}e^{-\lambda_jt},\qquad (\mathrm{Re}(t)>0).
\end{align}

On the other hand, similar to the Riemann zeta function, Koshliakov zeta function $\zeta_p(s)$ also does not seem to admit a nice closed form analogous to \eqref{zeta_p(2m)} at odd integers $s=2m+1$. In this direction, Dixit and Gupta \cite[Theorem 4.1]{dixitgupta1} provided an elegant symmetric formula for $\zeta_p(s)$ at odd integers. Namely, for $\alpha,\beta>0$ such that $\alpha\beta=\pi^2$,  we have
\begin{align}\label{dixitgupta formula}
&\a^{-m}\left\{\frac{1}{2}\zeta_{p}(2m+1)+\sum_{j=1}^{\infty}\frac{p^2+\l_j^2}{p\left(p+\frac{1}{\pi}\right)+\l_j^2}\cdot\frac{\l_{j}^{-2m-1}}{\s\left(\frac{\l_j \a}{\pi} \right)e^{2\a \l_j}-1}\right\}\nonumber\\
&=(-\b)^{-m}\left\{\frac{1}{2}\zeta_{p}(2m+1)+\sum_{j=1}^{\infty}\frac{p^2+\l_j^2}{p\left(p+\frac{1}{\pi}\right)+\l_j^2}\cdot\frac{\l_{j}^{-2m-1}}{\s\left(\frac{\l_j \b}{\pi} \right)e^{2\b \l_j}-1}\right\}\nonumber\\
&\quad -2^{2m}\sum_{j=0}^{m+1}(-1)^j\frac{B_{2j}^{(1,p)}B_{2m-2j+2}^{(1,p)}}{(2j)!(2m-2j+2)!}\alpha^{m+1-j}\beta^j,
\end{align}
where the function $\sigma(x)$ is defined as
\begin{align}\label{sigma}
	\sigma(x):=\frac{p+x}{p-x}\quad \ (p>0).
\end{align}

As $p\to\infty$, \eqref{dixitgupta formula} reduces to the celebrated formula of Ramanujan for $\zeta(2m+1)$ \cite{berndt nb, rlnb}:
\begin{align}\label{ramanujan formula}
\a^{-m}\left\{\frac{1}{2}\zeta(2m+1)+\sum_{n=1}^{\infty}\frac{n^{-2m-1}}{e^{2n\alpha-1}}\right\}
&=(-\b)^{-m}\left\{\frac{1}{2}\zeta(2m+1)+\sum_{n=1}^{\infty}\frac{n^{-2m-1}}{e^{2n\beta-1}}\right\}\nonumber\\
&\qquad-2^{2m}\sum_{j=0}^{m+1}(-1)^j\frac{B_{2j}B_{2m-2j+2}}{(2j)!(2m-2j+2)!}\alpha^{m+1-j}\beta^j.
\end{align}
This formula is remarkable as it encapsulates several fundamental results in the theory of the Riemann zeta function and modular forms. For instance, it yields, as special cases, the transformation formulas for Eisenstein series of even integral weight as well as for Eichler integrals. On the other hand, upon taking $\alpha=\beta=\pi$ and $m=2n+1$, \eqref{ramanujan formula} implies that either $\zeta(4n+3)$ or the series $\sum_{j=1}^{\infty}\frac{j^{-4m-3}}{e^{2j\pi-1}}$ is transcendental. For recent developments related to Ramanujan's formula and its significance, we refer the reader to the excellent survey article of Dixit \cite{dixitsurvey} and the references therein.

The recent literature on Koshliakov zeta functions has been more inclined towards studying $\zeta_p(s)$ rather than $\eta_p(s)$, in particular, analogues of \eqref{zeta_p(2m)} and \eqref{dixitgupta formula} do not yet seem to be available for $\eta_p(s)$. One possible reason for this preference might be that the latter is not defined as a `Dirichlet series', rendering its study more involved. In this paper, we focus on the second Koshliakov zeta function $\eta_p(s)$. We derive analogues of Euler's formula \eqref{zeta(2m)} (or of \eqref{zeta_p(2m)}), as well as of Ramanuja's formula \eqref{ramanujan formula} (or of \eqref{dixitgupta formula}) for $\eta_p(s)$, a very natural question that, however, remains unaddressed in the literature. We note here that our new  formulas are important for several reasons, which will become evident shortly. For instance, it finds applications in proving known results on transcendence of certain infinite series, as well as in defining the $p$-analogues of Eisenstein series and quasi-modular forms, together with their corresponding transformation formulas. Moreover, it provides a unifying framework under which many classical results due to mathematicians such as  Berndt, Cauchy, Ramanujan, and others can be obtained as special cases. For example, Berndt \cite[Theorem~3.1]{berndt Krelle} proved the following transformation formula:
\begin{align}\label{berndt formula}
&\alpha^{-m}\sum_{k=1}^{\infty}(-1)^k\frac{k^{-2m-1}}{e^{\alpha k}-e^{-\alpha k}}=(-\beta)^{-m}\sum_{k=1}^{\infty}(-1)^k\frac{k^{-2m-1}}{e^{k\beta}-e^{-k\beta }}\nonumber\\
&\qquad\qquad+2^{2m}\sum_{j=0}^{m+1}(-1)^{j+1}\left(2^{1-2j}-1\right)\left(2^{2j-2m-1}-1\right)\frac{B_{2j}B_{2m-2j+2}}{(2j)!(2m-2j+2)!}\alpha^j\beta^{m-j+1}.
\end{align}
The above result is quite useful and it leads to several interesting consequences, including closed-form evaluations of various infinite series such as
\begin{align}
\sum_{k=1}^{\infty}\frac{(-1)^kk^{-4m-3}}{e^{k\pi}-e^{-k\pi}}=\frac{1}{2\pi}\sum_{j=0}^{2m+2}(-1)^{j+1}\left(2^{1-2j}-1\right)\left(2^{2j-4m-3}-1\right)\frac{B_{2j}B_{4m-2j+4}}{(2j)!(4m-2j+4)!},\nonumber 
\end{align}
and 
\begin{align*}
\sum_{k=1}^{\infty}\frac{(-1)^kk^{4m+1}}{e^{k\pi}-e^{-k\pi}}=0,
\end{align*}
as well as  beautiful symmetric relations of the form, given by Cauchy \cite{cauchy},
\begin{align}\label{symm relation}
\alpha^{m+1}\sum_{k=1}^\infty(-1)^k\frac{k^{2m+1}}{e^{k\alpha}-e^{-k\alpha}}=(-\beta)^{m+1}\sum_{k=1}^\infty(-1)^k\frac{k^{2m+1}}{e^{k\beta}-e^{-k\beta}}.
\end{align}
 
In this paper, among other results, we show that the Ramanujan formula \eqref{ramanujan formula} and the Berndt formula \eqref{berndt formula} are natural companions of each other, in the sense that one arises as $p\to\infty$ and the other as $p\to0 $ from our new transformation formula for $\eta_p(s)$, given in Theorem \ref{etap at odd} below.

Apart from the applications of Ramanujan's formula \eqref{ramanujan formula} in the theory of the Riemann zeta function and modular forms, the finite sum on its right-hand side also possesses interesting properties and applications. More precisely, for $k\geq2$, Zagier \cite[p.~453, Proposition]{zagier} defined the rational functions
\begin{align}\label{ramanujan poly}
\mathscr{P}_{k}(z):=\sum_{j=0}^{2k}\frac{B_{j}B_{2k-j}}{j!(2k-j)!}z^{j-1}.
\end{align}
Observe that,  up to a factor of $1/z$, the above expression can be obtained from  the finite sum appearing in the Ramanujan's formula \eqref{ramanujan formula} by taking $\alpha=-\pi z $, $\beta=i\pi/z$ and replacing $k$ by $m+1-k$. 

These functions are very interesting and of utmost importance. For example, they are closely related to the period polynomials of Eisenstein series \cite[p.~4762]{cfi}. In the same paper \cite{zagier}, Zagier showed that $\mathscr{P}_{k}(z)$ satisfies the following striking two- and three-term functional equations (see also Vlasenko and Zagier \cite[p.~42]{vz})
\begin{align}\label{ram poly two term}
\mathscr{P}_{k}(z)-z^{2k-2}\mathscr{P}_{k}\left(\frac{1}{z}\right)=0,
\end{align}
and 
\begin{align}\label{ram poly three term}
\mathscr{P}_{k}(z)-\mathscr{P}_{k}(z-1)+z^{2k-2}\mathscr{P}_{k}\left(\frac{z-1}{z}\right)=0.
\end{align}
These rational functions have further applications, for instance, Vlasenko and Zagier \cite[p.~42, Corollary]{vz}) used them to obtain explicit formulas for the higher Kronecker limit formula for the zeta function of real quadratic fields. Moreover, Murty, Smyth and Wang \cite{msw} named the polynomials $z\mathscr{P}_{k}(z)$ as \emph{Ramanujan polynomials} and extensively studied their zeros and other properties.  

In this paper, we introduce $p$-analogues  of $\mathscr{P}_{k}(z)$ and establish their functional equations in Section \ref{p analogue of ramanujan poly}.

\section{New results for Koshliakov zeta function $\eta_p(s)$ and their consequences}
Euler's formula \eqref{zeta(2m)} can be obtained, using the functional equation \eqref{riemann fe} of $\zeta(s)$, from the following classical formula for the Riemann zeta function \cite[p.~605, Formula 25.63]{nist}
\begin{align}\label{zeta at 1-n}
\zeta(-m)=-\frac{1}{m+1}B_{m+1}, \qquad m\in\mathbb{N}.
\end{align}
In view of the functional equation of the Koshliakov zeta functions \eqref{kosh fe}, deriving an Euler-type formula for $\eta_p(2m)$ therefore requires establishing a $p$-analogue of \eqref{zeta at 1-n} for $\zeta_p(s)$, which is not yet available in the literature. This is the subject of our first result.
\begin{theorem}\label{even eta}
For $m\in\mathbb{N}$, we have
\begin{align}\label{zeta_p at 1-n}
\zeta_p(-m)=-\frac{1}{m+1}B_{m+1}^{(2,p)}.
\end{align}
Consequently, we have the following Euler-type formula for $\eta_p(2m)$:
\begin{align}\label{even eta eqn}
\eta_p(2m) &= \frac{(-1)^{m+1} (2\pi)^{2m}}{2 (2m)!} B_{2m}^{(2,p)},
\end{align}
where $B_{2m}^{(2,p)}$  are the generalized Bernoulli numbers of the second kind which are defined by the generating function \cite[p.~43, Chapter 2, Equation (31)]{koshliakov}
\begin{align}\label{B2m2p}
\sum_{n=0}^{\infty} \frac{B_n^{(2,p)}}{n!} t^n = t \sum_{j=1}^{\infty} \frac{p^2 + \lambda_j^2}{p(p + \frac{1}{\pi}) + \lambda_j^2} e^{-\lambda_j t}=t\sigma_p(t)\qquad(0<t<1),
\end{align}
where $\sigma_p(t)$ is defined in \eqref{sigma_p}.
\end{theorem}
Note that, as $p\to\infty$, equation \eqref{B2m2p} reduces to the generating function of the classical Bernoulli numbers
\begin{align}
\sum_{n=0}^{\infty} \frac{B_n}{n!} t^n = \frac{t}{e^t-1}. \nonumber
\end{align}
Therefore, Euler's formula \eqref{zeta(2m)} and \eqref{zeta at 1-n} follow as special cases of \eqref{even eta eqn} and \eqref{zeta_p at 1-n}, respectively.

Our next proposition provides an integral representation for $B_{2m}^{(2,p)}$ similar to \eqref{bernoulli p1}.
\begin{proposition}\label{B2m2p intergal representation}
Let $m\in\mathbb{N}$, we have
\begin{align}\label{B2m2p intergal representation eqn}
B_{2m}^{(2,p)}= 4m (-1)^{m+1} \int_{0}^{\infty} \frac{x^{2m-1}}{\sigma(x)e^{2\pi x} - 1} \, dx,\qquad B_{0}^{(2,p)}=1,
\end{align}
with the function $\sigma(x)$ is defined in \eqref{sigma}.
\end{proposition}

The cases $p\to0$ and $p\to\infty$ of \eqref{B2m2p intergal representation eqn} reduce to the classical Bernoulli numbers as shown in the next corollary.
\begin{corollary}\label{bernoulli special case}
	Let $m\in\mathbb{N}$. Then
	\begin{align}\label{bernoulli special case p0}
	B_{2m}^{(2,0)}=\left(2^{1-2m}-1\right)B_{2m},
	\end{align}
	and 
	\begin{align}\label{bernoulli special case p infty}
		\lim_{p\to\infty}B_{2m}^{(2,p)}=B_{2m}.
	\end{align}
\end{corollary}


We now shift our focus from the special values of $\eta_p(s)$ at even integers to those at odd integers. To that end, the first and crucial step is to identify the series corresponding to
\begin{align}\label{two series}
\sum_{n=1}^\infty\frac{n^{-2m-1}}{e^{2nx}-1}\qquad\mathrm{or}\qquad \sum_{j=1}^{\infty}\frac{p^2+\l_j^2}{p\left(p+\frac{1}{\pi}\right)+\l_j^2}\frac{\l_{j}^{-2m-1}}{\s\left(\frac{\l_j \a}{\pi} \right)e^{2\a \l_j}-1}.
\end{align}
To do that, we first define $p$-analogue of the exponential function as follows
\begin{align}\label{exp defn}
\exp_p(-x;z,k):=\sum_{n=0}^{\infty} \frac{(-1)^n}{n!} (z-n, 2\pi pk)_k x^n,
\end{align}
where $|x|<2\pi p k,\ z\in\mathbb{R}$, $k\in\mathbb{N}$, and $(s,\nu k)_k$ is defined through equation \eqref{gamma generalized} and its analytic continuation in Proposition \ref{bracket is entire}. The absolute convergence of the series in \eqref{exp defn} for $|x|<2\pi p k$ is shown in proposition \ref{exp_p mellin} below.
Note that, as $p\to\infty$, the $p$-exponential function $\exp_p(-x;z,k)$, in view of \eqref{bracket}, reduces to the classical exponential function $\exp(-x)$, namely,
\begin{align}\label{exp p infty}
\lim_{p\to\infty}\exp_p(-x;z,k)=\exp(-x).
\end{align} 

We also define the function
\begin{align}\label{omega function}
\Omega_p(x,m):=\sum_{k=1}^\infty\frac{\exp_p(-xk;2m+1,k)}{k^{2m+1}},
\end{align}
where $m\in\mathbb{N}$ and $\mathrm{Re}(x)>0$. This function serves as the key ingredient for the required analogue of the series  in \eqref{two series} and will play a crucial role in what follows.

With these notations in place, we are now ready to state our next result.

\begin{theorem}\label{etap at odd}
Let\footnote{The $m=0$ case is covered in Theorem \ref{case m=0} below.} $m\in \mathbb{Z}/\{0\}$. Then for $\mathrm{Re}(\alpha),\mathrm{Re}(\beta)>0$ such that $\alpha\beta=\pi^2$, the following identity holds:
\begin{align}\label{etap at odd eqn}
&\alpha^{-m}\left\{\frac{1}{2}\frac{1}{1+\frac{1}{\pi p}} \eta_p(2m+1)+\sum_{j=1}^{\infty}\frac{(p^2+\lambda_j^2)\Omega_p(2\lambda_j\alpha,m)}{p\left(p+\frac{1}{\pi }\right)+\lambda_j^2}\right\}\nonumber\\
&=(-\beta)^{-m}\left\{\frac{1}{2}\frac{1}{1+\frac{1}{\pi p}} \eta_p(2m+1)+\sum_{j=1}^{\infty}\frac{(p^2+\lambda_j^2)\Omega_p(2\lambda_j\beta,m)}{p\left(p+\frac{1}{\pi }\right)+\lambda_j^2}\right\}\nonumber\\
	&\qquad\qquad-2^{2m}\sum_{j=0}^{m+1}(-1)^j\frac{B_{2j}^{(2,p)}B_{2m-2j+2}^{(2,p)}}{(2j)!(2m-2j+2)!}\alpha^{m+1-j}\beta^j.
\end{align}
\end{theorem}
\begin{remark}
Note that the above result yields an infinite family of transformation formulas as it holds for every $p\in(0,\infty)$. Interestingly, the boundary cases, namely, $p\to0$ and $p\to\infty$, correspond precisely to Berndt's formula \eqref{berndt formula} and Ramanujan's formula \eqref{ramanujan formula} for $\zeta(2m+1)$, respectively, as shown in the following two corollaries.
\end{remark}

\begin{corollary}\label{ram coro}
Ramanujan's formula for $\zeta(2m+1)$,  given in \eqref{ramanujan formula} above, holds true.
\end{corollary}

\begin{corollary}\label{berndt coro}
	The transformation formula of Berndt, given in \eqref{berndt formula}  above, holds true.
\end{corollary}

The next result of this section is the following symmetric relation under the map $\alpha\to\beta$.
\begin{theorem}\label{t2}
For $m\in\mathbb{N}$ and $\mathrm{Re}(\alpha),\mathrm{Re}(\beta)>0$ such that $\alpha\beta=\pi^2$, we have
\begin{align}\label{t21}
\nonumber 
&\alpha^{m+1}\biggr\{\sum_{j=1}^{\infty}\frac{p^2+\lambda_j^2}{p\left(p+\frac{1}{\pi }\right)+\lambda_j^2}\Omega_p(2\lambda_j\alpha,-m-1)+
\frac{1}{2}\frac{1}{1+\frac{1}{\pi p}} \eta_p(-2m-1)\biggr\}\\
&=\left( -\beta \right)^{m+1} 
\biggr\{
\sum_{j=1}^{\infty}\frac{p^2+\lambda_j^2}{p\left(p+\frac{1}{\pi }\right)+\lambda_j^2}\Omega_p(2\lambda_j\beta,-m-1)+
\frac{1}{2} \frac{1}{1+\frac{1}{\pi p}}  \eta_p(-2m-1) \biggr\}.
\end{align}
\end{theorem}
Note that the case $p\to0$ of the above result reduces to a result of Cauchy, given in \eqref{symm relation}, due to \eqref{lambert series special case p0}.


\subsection{Generalized Ramanujan polynomials}\label{p analogue of ramanujan poly}

In the view of rational functions $\mathscr{P}_{k}(z)$ in \eqref{ramanujan poly}, it is now natural to consider the $p$-analogues of $\mathscr{P}_{k}(z)$  appearing in \eqref{dixitgupta formula} and our Theorem  \ref{etap at odd}. Accordingly, we define
\begin{align}\label{p-ramanujan poly}
\mathscr{P}^{(\ell,p)}_{k}(z):=\sum_{j=0}^{2k}\frac{B_{j}^{(\ell,p)}B_{2k-j}^{(\ell,p)}}{j!(2k-j)!}z^{j-1},
\end{align}
for $k\in\mathbb{N}$ and $\ell=1,2$.

We note that the case $\ell=1$ of the above functions does not appear to have been considered previously in the present context, including in \cite{dixitgupta1}.

Our next result provides functional equations for  both $\mathscr{P}^{(1,p)}_{k}(z)$ and $\mathscr{P}^{(2,p)}_{k}(z)$.
\begin{theorem}\label{etap functional equation}
Let $\ell=1,2$ and $k>2$. Then for $z\in\mathbb{C}$, we have
\begin{align}\label{etap functional equation two term}
\mathscr{P}_{k}^{(\ell,p)}(z)-z^{2k-2}\mathscr{P}_{k}^{(\ell,p)}\left(\frac{1}{z}\right)=0,
\end{align}
and 
\begin{align}\label{etap functional equation three term}
\mathscr{P}_{k}^{(\ell,p)}(z)-\mathscr{P}_{k}^{(\ell,p)}(z-1)+z^{2k-2}\mathscr{P}_{k}^{(\ell,p)}\left(\frac{z-1}{z}\right)=\mathcal{E}^{(\ell,p)}(z),
\end{align}
where $\mathcal{E}^{(\ell,p)}(x)$ denotes the coefficient of  $t^{2k}$ in the power series expansion of
\begin{align}\label{ep1}
\mathcal{E}^{(1,p)}(x,t):=\frac{x(1-x)t^{5} e^{xt}}{4\pi^{3} f_p(t)f_p(xt)f_p((x-1)t)}, \quad f_p(y):=\left(e^y-1
\right)\left(p+\frac{y}{2\pi}\right)+\frac{y}{\pi},
\end{align}
and 
\begin{align}\label{ep2}
\mathcal{E}^{(2,p)}(x,t):=t^2\bigl\{\sigma_p(xt)\sigma_p(t)-\sigma_p\left((x-1)t\right)\sigma_p(t)+\sigma_p\left((x-1)t\right)\sigma_p(xt)
\bigl\}.
\end{align}
\end{theorem}


The case $p\to\infty$ of the above theorem reduces to \eqref{ram poly two term} and \eqref{ram poly three term} as shown in the following corollary.
\begin{corollary}\label{ram poly as a special case}
The period relations for the Ramanujan polynomials, given in \eqref{ram poly two term} and \eqref{ram poly three term}, hold true.
\end{corollary}

\subsection{Closed-form evaluations of certain Infinite Series  arising from the Transformation Formula for $\eta_p(s)$}

Now what follows are several interesting new results together with some known results from the literature that appear as special cases  of Theorem \ref{etap at odd}.

\begin{theorem}\label{even integer}
For a positive even integer $m$, we have
\begin{align}
\sum_{j=1}^{\infty}\frac{p^2+\lambda_j^2}{p\left(p+\frac{1}{\pi }\right)+\lambda_j^2}\Omega_p(2\lambda_j\pi,-m-1)=-\frac{1}{2} \frac{1}{1+\frac{1}{\pi p}}  \eta_p(-2m-1).\nonumber
\end{align}
\end{theorem}

As $p\to0$ and $p\to\infty$, the above theorem yields the following elegant results due to Cauchy \cite[p.~310]{cauchy} and Glaisher \cite{glaiser}.
\begin{corollary}\label{closed form is zero}
For $m\in\mathbb{N}$,
\begin{align}\label{closed form is zero eqn}
\sum_{k=1}^{\infty}(-1)^k\frac{k^{4m+1}}{e^{k\pi}-e^{-k\pi}}=0,
\end{align}
and
\begin{align}\label{closed form is zero eqn11}
\sum_{k=1}^{\infty}\frac{k^{4m+1}}{e^{2\pi k}-1}=-\frac{1}{2}\zeta(-2m-1)=\frac{B_{2m+2}}{4m+2}.
\end{align}
\end{corollary}

As a special case of Theorem \ref{etap at odd}, the next result is as follows.
\begin{theorem}\label{t4}
For $m\in\mathbb{Z}/ \{-1\}$ and $p>0$,
\begin{align}
2\sum_{j=1}^{\infty}\frac{p^2+\lambda_j^2}{p\left(p+\frac{1}{\pi }\right)+\lambda_j^2}\Omega_p(2\lambda_j\pi,2m+1)=\frac{-1}{1+\frac{1}{\pi p}} \eta_p(4m+3)+\frac{(2\pi)^{4m+3}}{2}\sum_{j=0}^{2m+2}\frac{(-1)^{j+1}B_{2j}^{(2,p)}B_{4m-2j+4}^{(2,p)}}{(2j)!(4m-2j+4)!}.\nonumber
\end{align}
\end{theorem}

The following classical results due to Lerch \cite{lerch} and Cauchy \cite[p.~311--313 and p.~361]{cauchy} follow when $p\to\infty$ and $p\to0$ in the above theorem.
\begin{corollary}\label{cauchy}
Let $m$ be any natural number. Then
\begin{align}\label{lerch1}
	\zeta(4m+3)=2^{4m+2}\pi^{4m+3}\sum_{j=0}^{2m+2}(-1)^{j+1}\frac{B_{2j}B_{4m+4-2j}}{(2j)!(4m+4-2j)!}-2\sum_{n=1}^\infty\frac{n^{-4m-3}}{e^{2\pi n}-1},
\end{align}
and 
\begin{align}\label{cauchy eqn}
\sum_{k=1}^{\infty}(-1)^k\frac{k^{-4m-3}}{e^{k\pi}-e^{-k\pi}}&=\frac{(2\pi)^{4m+3}}{4}\sum_{j=0}^{2m+2}(-1)^{j+1}\left(2^{1-2j}-1\right)\left(2^{2j-4m-3}-1\right)\frac{B_{2j}B_{4m-2j+4}}{(2j)!(4m-2j+4)!},
\end{align}
Consequently, the Lambert series on the left-hand side of \eqref{cauchy eqn} is transcendental.
\end{corollary}

The next result gives the closed-form evaluation of the infinite series obtained from Theorem \ref{etap at odd} by taking $m=-1$ and $\alpha=\beta=\pi$.
\begin{theorem}\label{closed form}
We have
\begin{align*}
\sum_{j=1}^{\infty}\frac{p^2+\lambda_j^2}{p\left(p+\frac{1}{\pi }\right)+\lambda_j^2}\Omega_p(2\pi\lambda_j,-1)=-
\frac{1}{2}\frac{1}{1+\frac{1}{\pi p}} \eta_p(-1)-\frac{1}{8\pi}.
\end{align*}
\end{theorem}
The boundary cases $p\to\infty$ and $p\to0$ of the above theorem lead to the following elegant results obtained by Cauchy \cite[p.~361]{cauchy}.
\begin{corollary}\label{last cor}
\begin{align}\label{last cor1}
\sum_{k=1}^{\infty}\frac{k}{e^{2\pi k}-1}=\frac{1}{24}-\frac{1}{8\pi},
\end{align}
and
\begin{align}\label{last cor2}
\sum_{k=1}^{\infty}(-1)^{k+1}\frac{k}{e^{\pi k}-e^{-\pi k }}=\frac{1}{8\pi}.
\end{align}
Therefore, the series on the left-hand sides of \eqref{last cor1} and \eqref{last cor2} are transcendental.
\end{corollary}

We end this subsection with the following remark.
\begin{remark}
We emphasize that the results in \eqref{dixitgupta formula} and \eqref{ramanujan formula} for $\zeta_p(s)$ and $\zeta(s)$ respectively do not yield the closed-form expressions of the type given in  \eqref{closed form is zero} and \eqref{cauchy eqn}, which in turn establish important results on the transcendence of certain Lambert series that may be difficult to prove directly. This is due to the presence of the term $\frac{1}{1+\frac{1}{\pi p}}$ in the transformation formula \eqref{etap at odd eqn}, which forces the term containing $\eta_p(2m+1)$ to vanish in the limit $p\to0$. Such a phenomenon does not occur for  $\zeta_p(s)$ and $\zeta(s)$, and hence the corresponding closed forms cannot be obtained from the formulas \eqref{dixitgupta formula} and \eqref{ramanujan formula}.
\end{remark}

\subsection{$p$-analogue of the Eisenstein series}
In this section, we define $p$-analogue of Eisenstein series. Note that the classical Eisenstein series  of  weight $k>1$ over SL$_2(\mathbb{Z})$ can be defined by the following Fourier
series expansion
\begin{align}\label{eisen}
E_k(z)=1+\frac{2}{\zeta(1-k)}\sum_{n=1}^\infty\frac{n^{k-1}}{e^{-2\pi inz}-1} \qquad (z\in\mathbb{H}),
\end{align}
where $\mathbb{H}$ is  the Poincare upper half-plane $\mathbb{H}:=\{z\in\mathbb{C}:\ \mathrm{Im}(z)>0\}$. It is well-known that $E_k(z)$, with $k$ being a positive even number, satisfies the following transformations over SL$_2(\mathbb{Z})$:
\begin{align}\label{modulal trans}
(i)\quad  E_k(z+1)&=E_k(z)\nonumber\\
(ii)\quad E_k\left(-\frac{1}{z}\right)&=z^kE_k(z),
\end{align}
 hence, it is a modular form of weight $k$ over SL$_2(\mathbb{Z})$.

We next define $p$-analogue of $E_k(z)$. For $k\geq2$ and $k\in\mathbb{N}$, let use define\footnote{If $\eta_p(1-k)=0$, then we define $E_k^{(p)}(z)$ by multiplying the right-hand side by $\eta_p(1-k)$.}
\begin{align}\label{p-eisen}
E_k^{(p)}(z):=1+
\frac{2(1+\frac{1}{\pi p})}{ \eta_p(1-k)}\sum_{j=1}^{\infty}\frac{p^2+\lambda_j^2}{p\left(p+\frac{1}{\pi }\right)+\lambda_j^2}\Omega_p\left(-2\lambda_j\pi iz,-\frac{k}{2}\right),
\end{align}
where $z\in\mathbb{H}$.

Note that as $p\to\infty$, invoking \eqref{lambert series special case}, $E_k^{(p)}(z)$ reduces to the classical Eisenstein series $E_k(z)$ in \eqref{eisen}.

Our next result shows that $p$-analogue of the Eisenstein series satisfies the transformation
formula \eqref{modulal trans} of an integral weight modular form over SL$_2(\mathbb{Z})$.

\begin{theorem}\label{eisent trans}
Let $k$ be any natural number greater than 1 and $z\in\mathbb{H}$, we have
\begin{align*}
E_{2k}^{(p)}\left(-\frac{1}{z}\right)=z^{2k}E_{2k}^{(p)}(z).
\end{align*}
\end{theorem}

The Eisenstein series $E_2(z)$ of weight 2 is referred to as a quasi-modular form because it does not satisfy the modular transformation \eqref{modulal trans} exactly. Instead, it obeys relation
\begin{align}\label{e2}
E_2\left(-\frac{1}{z}\right)=z^kE_k(z)+\frac{6z}{\pi i}.
\end{align} 

Our next theorem generalizes this quasi-modular relation in the $p$-setting.
\begin{theorem}\label{quasi trans}
For $z\in\mathbb{H}$, we have
\begin{align}
E_2^{(p)}\left(-\frac{1}{z}\right)=z^2E_2^{(p)}(z)-\frac{z}{2\pi i}\frac{1}{\eta_p(-1)}\left(1+\frac{1}{\pi p}\right).\nonumber
\end{align}
\end{theorem}
The above result reduces to \eqref{e2} in the case $p\to\infty$ as $\lim_{p\to\infty}\eta_p(-1)=\zeta(-1)=-1/12$.

The final result of this subsection generalizes the transformation formula \cite[p.~320, Formula (3.6)]{rlnb}
\begin{align}\label{log eta}
\sum_{j=1}^\infty\frac{1}{j(e^{2j\alpha}-1)}-\sum_{j=1}^\infty\frac{1}{j(e^{2j\beta}-1)}=\frac{\beta-\alpha}{12}+\frac{1}{4}\log\frac{\alpha}{\beta},
\end{align}
 which is the transformation formula for the logarithm of the Dedekind eta function.
 \begin{theorem}\label{case m=0}
 For $\alpha\beta=\pi^2$, we have
 \begin{align}
\sum_{j=1}^{\infty}\frac{p^2+\lambda_j^2}{p\left(p+\frac{1}{\pi }\right)+\lambda_j^2}\Omega_p(2\lambda_j\alpha,0)-&\sum_{j=1}^{\infty}\frac{p^2+\lambda_j^2}{p\left(p+\frac{1}{\pi }\right)+\lambda_j^2}\Omega_p(2\lambda_j\beta,0)\nonumber
=\frac{B_2^{(2,p)}}{2}(\beta-\alpha)+\frac{1}{4}\frac{\log\left(\alpha/\beta\right)}{\left(1+\frac{1}{\pi p}\right)^2}.\nonumber
\end{align}
 \end{theorem}
The case $p\to\infty$ is nothing but \eqref{log eta}, whereas the case $p\to0$ yields the result of Cauchy \cite[p.~337, Corollary 3.3]{berndt Krelle}:
\begin{corollary}\label{cauchy10}
Let $\alpha,\beta>0$ such that $\alpha\beta=\pi^2$. Then
\begin{align}
\sum_{k=1}^{\infty}\frac{(-1)^k}{k\left(e^{\alpha k}-e^{-\alpha k}\right)}-\sum_{k=1}^{\infty}\frac{(-1)^k}{k\left(e^{\beta k}-e^{-\beta k}\right)}
=\frac{1}{24}(\beta-\alpha).\nonumber
\end{align}
\end{corollary}

\section{Preliminaries and basic properties of $p$-analogues of some classical functions}

Note that the function $(s,\nu k)_k$, defined in \eqref{gamma generalized}, is valid only for Re$(s)>0$, whereas in the definition of $\exp_p$-function in \eqref{exp defn}, we require its values at negative arguments of $s$ as well. Therefore, it is imperative to obtain an analytic continuation of $(s,\nu k)_k$ beyond the region Re$(s)>0$. This is achieved in the next proposition by deriving  a new representation for it.
\begin{proposition}\label{bracket is entire}
\begin{enumerate}
\item Let $\nu>0$ and $k\in\mathbb{N}$. For $\mathrm{Re}(s)>0$, we have
\begin{align}\label{bracket analytic}
(s,\nu k)_k=\frac{(\nu k)^s}{\Gamma(s)}\sum_{j=0}^k(-1)^j\binom{k}{j}\Gamma(s+j)U(s+j;s+j+1-k;k\nu),
\end{align}
where $U(a;c;z)$ is the Kummer's hypergeometric function defined as \cite[p.~175, Equation (7.12)]{temme}
\begin{align}\label{U function}
U(a;c;z):=\frac{1}{\Gamma(a)}\int_0^\infty t^{a-1}(1+t)^{c-a-1}e^{-zt}dt \quad (\mathrm{Re}(a), \mathrm{Re}(z)>0). 
\end{align}
  \item Consequently,  the  function $(s, \nu k)_k$ can be analytically continued in $s$ to the entire complex plane without any singularities.
\end{enumerate}
\end{proposition}
\begin{proof}
We start with the definition of $(s, \nu k)_k$ in \eqref{gamma generalized} and observe that
\begin{align}\label{an cont}
(s, \nu k)_k&=\frac{1}{\Gamma(s)}\int_0^\infty e^{-x}\left(1-\frac{x}{\nu k}\right)^k\left(1-\frac{x}{\nu k}\right)^{-k}x^{s-1}dx\nonumber\\
&=\frac{1}{\Gamma(s)}\int_0^\infty e^{-x}\left(1+\frac{x}{\nu k}\right)^{-k}x^{s-1}\sum_{j=0}^k(-1)^j\binom{k}{j}\left(\frac{x}{\nu k}\right)^jdx\nonumber\\
&=\frac{1}{\Gamma(s)}\sum_{j=0}^k(-1)^j\binom{k}{j}\left(\nu k\right)^{-j}\int_0^\infty e^{-x}\left(1+\frac{x}{\nu k}\right)^{-k}x^{s+j-1}dx.
\end{align}         
Employing the change of variable $t=\frac{x}{\nu k}$ in \eqref{U function}, we see that
\begin{align*}
U(a;c;z)=\frac{(\nu k)^{-a}}{\Gamma(a)}\int_0^\infty e^{-\frac{zx}{\nu k}}x^{a-1}\left(1+\frac{x}{\nu k}\right)^{c-a-1}dx.
\end{align*}
We now invoke the above result with $z=\nu k,\ a=s+j$ and $c=s+j+1-k$ and then substitute the resulting expression in \eqref{an cont} to arrive at \eqref{bracket analytic}.

To prove the second part, using the functional equation $\Gamma(s+1)=s\Gamma(s)$, we rewrite \eqref{bracket analytic} as
\begin{align}\label{modified exp}
(s,\nu k)_k=(\nu k)^s\sum_{j=0}^k(-1)^j\binom{k}{j}\left\{(s+j-1)(s+j-2)\cdots(s)\right\} U(s+j;s+j+1-k;k\nu).
\end{align}
Since the function $U(a;c;z)$ is entire in $a$ and $c$ when $z\neq0$ \cite[p.~323, Section 13.2(ii)]{nist},  and other functions in \eqref{modified exp} are entire in $s$, it follows that the right-hand side of  \eqref{bracket analytic}  is an entire function of $s$. This completes the proof of the second part.

\end{proof}          
                     
                     
 We would be using the following  Stirling's formulas of $\Gamma(\sigma+it)$ in a vertical strip $a\leq\sigma\leq b$ \cite[p.~224]{copson}
\begin{align}\label{stirling}
|\Gamma(s)|=\sqrt{2\pi}|t|^{\frac{1}{2}-\sigma}e^{-\frac{1}{2}\pi t}\left(1+\mathcal{O}\left(\frac{1}{|t|}\right)\right),
\end{align}          
as $|T| \to \infty$, and  \cite[p.~140, Formula 5.11.3]{nist}
\begin{align}\label{striling2}
	\Gamma(z)\sim e^{-z}z^z\left(\frac{2\pi}{z}\right)^{1/2}, \quad z\to\infty\ \mathrm{in}\ |\arg (z)|<\pi,
\end{align}
                     
\section{Proofs}     
\subsection{Euler's formula for the second Koshliakov zeta function $\eta_p(2m)$}

\begin{proof}[\textbf{Theorem \textup{\ref{even eta}}}][]
We start with the generating function of $B_n^{(2,p)}$, given in \eqref{B2m2p},
\begin{align*}
\sum_{n=0}^{\infty} \frac{B_n^{(2,p)}}{n!} t^n = t \sum_{j=1}^{\infty} \frac{p^2 + \lambda_j^2}{p(p + \frac{1}{\pi}) + \lambda_j^2} e^{-\lambda_j t},
\end{align*}
and use the standard result
\begin{align*}
e^{-x}=\frac{1}{2\pi i}\int_{(c)}\Gamma(s)x^{-s}ds,\quad c>0,
\end{align*}
to arrive at (here, and throughout the paper, $\int_{(\lambda)}ds$ denotes the line integral $\int_{\lambda-i\infty}^{\lambda+i\infty}$ with $\lambda=\mathrm{Re}(s)$)
\begin{align}\label{bn eqn}
\sum_{n=0}^{\infty} \frac{B_n^{(2,p)}}{n!} t^n &= t \sum_{j=1}^{\infty} \frac{p^2 + \lambda_j^2}{p(p + \frac{1}{\pi}) + \lambda_j^2} \frac{1}{2\pi i}\int_{(c)}\Gamma(s)(t\lambda_j )^{-s}ds\nonumber\\
&=\frac{1}{2\pi i}\int_{(c)}\Gamma(s)\zeta_p(s)t^{1-s}ds,
\end{align}
where we interchanged the order of summation and integration and then invoked the definition of $\zeta_p(s)$ from \eqref{zetap defn}. To evaluate the right-hand side, we wish to move the line of integration from Re$(s)=c$ to Re$(s)=-d$ where $d$ is a positive non-integer real number. To that end, consider the rectangular region with line segments $[c-iT,c+iT],\  [c+iT,-d+iT],\ [-d+iT,-d-iT]$ and $[-d-iT,c-iT]$. Note that the integrand has simple poles at $s=-n, 0\leq n\leq \lfloor d\rfloor$ with residue $t^{n+1}\zeta_p(-n)(-1)^n/n!$, and a simple pole at $s=1$ with residue $1$. Hence, Cauchy residue theorem implies that
\begin{align}\label{bn eqn 1}
\frac{1}{2\pi i}\left(\int_{c-iT}^{c+iT}-\int_{-d-iT}^{-d+iT}+\int_{c+iT}^{-d+iT}+\int_{-d-iT}^{c-iT}\right)\Gamma(s)\zeta_p(s)t^{1-s}ds=1+\sum_{n=0}^{\lfloor d\rfloor}\frac{(-1)^n\zeta_p(-n)}{n!}t^{n+1}.
\end{align}
Note that application of Stirling formula \eqref{stirling} and  \cite[p.~24, Chapter 1, Equation (44)]{koshliakov}
\begin{align*}
	\zeta_p(\sigma+it)=\mathcal{O}\left(|t|^a\log|t|\right), \quad |t|\to\infty,
\end{align*}
show that the integrals along horizontal lines vanish as $T\to\infty$. Consequently, from \eqref{bn eqn} and \eqref{bn eqn 1}, we deduce that
\begin{align}\label{bn eqn 2}
\sum_{n=0}^{\infty} \frac{B_n^{(2,p)}}{n!} t^n &=1+\sum_{n=0}^{\lfloor d\rfloor}\frac{(-1)^n\zeta_p(-n)}{n!}t^{n+1}+\frac{1}{2\pi i}\int_{(-d)}\Gamma(s)\zeta_p(s)t^{1-s}ds.
\end{align} 
We next show that the integral in the above expression goes to zero as $d\to\infty$. To prove this, we first make the change of variable $s=-d+iu$ in the integral and then employ \eqref{striling2} and \cite[p.~24, Equation (43)]{koshliakov}
\begin{align*}
\zeta_p(\sigma+iu)=\mathcal{O}\left(u^{\frac{1}{2}-\sigma}\right),\qquad (\sigma<0)
\end{align*}
to deduce that, as $d\to\infty$,
\begin{align}
\left|\frac{1}{2\pi i} \int_{(-d)}\Gamma(s) \zeta_p(s)  t^{1-s}ds\right|
&=\left|\frac{1}{2\pi} \int_{-\infty}^\infty\Gamma(-d+iu) \zeta_p(-d+iu)  t^{1+d-iu}du\right|\nonumber\\
&=|t|^{d+1}\int_{-M}^M\mathcal{O}(1)du+|t|^{d+1}\int_{|u|\geq M}\mathcal{O}\left(|u|^{\alpha-\frac{1}{2}}e^{-\frac{\pi}{2}|u|}\right)du\nonumber\\
&\ll |t|^{d+1},\nonumber
\end{align}
where $M$ is a sufficiently large positive real number, $\alpha$ is some constant,  and we used the fact that both the integrals on the right-hand side are finite. Assume $|t|<1$,  it follows that $|t|^{d+1}\to0$ as $d\to\infty$, and hence proving that 
$$\int_{(-d)}\Gamma(s)\zeta_p(s)t^{1-s}ds\to0,\ \mathrm{as}\ d\to\infty.$$
Therefore, equation \eqref{bn eqn 2} simplifies to
\begin{align*}
\sum_{n=0}^{\infty} \frac{B_n^{(2,p)}}{n!} t^n &=1-\sum_{n=1}^{\infty}\frac{(-1)^{n}\zeta_p(1-n)}{(n-1)!}t^{n}.
\end{align*}
Now comparing the coefficients of $t^n$ on both sides of the above expression, we arrive at \eqref{zeta_p at 1-n}.

Equation \eqref{even eta eqn} follows directly from \eqref{zeta at 1-n} together with the functional equation \eqref{kosh fe} after simplifying the resulting expressions.
\end{proof}

\begin{proof}[\textbf{Proposition \textup{\ref{B2m2p intergal representation}}}][]
	Letting $s=2m\  (m\in \mathbb{N})$ in \eqref{kosh fe}, we have
	\begin{align}\label{before integral}
	\eta_p(2m) = \frac{(-1)^m (2\pi)^{2m}}{2 (2m-1)!} \zeta_p(1-2m).
	\end{align}
	We now make use of the following integral representation for $\zeta_p(s)$ \cite[Chapter~1, Equation~(74)]{koshliakov},  valid for $\mathrm{Re}(s)> 1$, 
	\begin{equation}
	    \zeta_p(1-s) = 2 \cos\left(\frac{\pi s}{2}\right) \int_0^\infty \frac{x^{s-1}}{\sigma(x)e^{2\pi x}-1} \, dx.\nonumber
	\end{equation}
	Replacing $s$ by $2m$ in the above result and then using it in \eqref{before integral}, we deduce that
	\begin{align}
	\eta_p(2m)=\frac{(2\pi)^{2m}}{(2m-1)!}\int_0^\infty \frac{x^{2m-1}}{\sigma(x)e^{2\pi x}-1} \, dx.\nonumber
	\end{align}
	Employing \eqref{even eta} in the above equation and simplifying the expression, we arrive at \eqref{B2m2p intergal representation eqn}.
\end{proof}

\begin{proof}[\textbf{Corollary \textup{\ref{bernoulli special case}}}][]
Letting $p=0$ in \eqref{B2m2p} and using the fact that $\sigma(x)|_{p=0}=-1$ from \eqref{sigma}, we see that
\begin{align}\label{before zeta}
B_{2m}^{(2,0)}=4m (-1)^{m} \int_{0}^{\infty} \frac{x^{2m-1}}{e^{2\pi x} +1} dx.
\end{align}
We next make use of the following integral representation of the Riemann zeta function \cite[p.~604, Formula (25.5.3)]{nist}
\begin{align*}
\int_0^\infty\frac{x^{s-1}}{e^{2\pi x} +1}dx =\left(1-2^{1-s}\right)\Gamma(s)\zeta(s), \quad \textup{Re(}s\textup{)}>0.
\end{align*}
Replacing $s=2m$ in the above result and substituting it in \eqref{before zeta}, we deduce that
\begin{align}
B_{2m}^{(2,0)}=4m (-1)^{m} \left(1-2^{1-2m}\right)\Gamma(2m)\zeta(2m).\nonumber
\end{align}
Invoking \eqref{zeta(2m)} on the right-hand side, we complete the proof of \eqref{bernoulli special case p0}.

We next prove \eqref{bernoulli special case p infty}. To that end, we let $p\to\infty$ in \eqref{B2m2p} and use the fact that $\lim_{p\to\infty}\sigma(x)=1$ to see that
\begin{align}
\lim_{p\to\infty}B_{2m}^{(2,p)}=4m (-1)^{m+1} \int_{0}^{\infty} \frac{x^{2m-1}}{e^{2\pi x} -1} dx.\nonumber
\end{align}
We now use the well-known identity for $\zeta(s)$ given as \cite[p.~604, Formula (25.5.1)]{nist}
\begin{align*}
	\int_0^\infty\frac{x^{s-1}}{e^{2\pi x} -1}dx =\Gamma(s)\zeta(s), \quad \textup{Re(}s\textup{)}>1,
\end{align*}
together with \eqref{zeta(2m)} to arrive at \eqref{bernoulli special case p infty}.
\end{proof}

\subsection{Transformation formula for the second Koshliakov zeta function $\eta_p(s)$ at odd integers}

We begin by obtaining the following proposition, which provides several properties of the $\exp_p$-function, including its inverse Mellin transform and analytic continuation.
\begin{proposition}\label{exp_p mellin}
	\begin{enumerate}
		\item Let $z\in\mathbb{R}$ and $k\in\mathbb{N}$. The series in \eqref{exp defn} is absolutely convergent for  $x\in\{w\in\mathbb{C}:|w|<2\pi pk\}$.
\item Let $|x|<2\pi pk$ and $k\in\mathbb{N}$. Then, for any $z\in\mathbb{R}$, we have
\begin{align}\label{exp_p mellin eqn}
\exp_p(-x; z, k)=\frac{1}{2\pi i} \int_{(c)}\Gamma(s) (s+z, 2\pi pk)_k  x^{-s}ds \qquad (c>0).
\end{align}
\item The right-hand side of \eqref{exp_p mellin eqn} provides analytic continuation of $\exp_p(-x)$ in the region $x\in\mathbb{C}\backslash(-\infty,0]$.
\end{enumerate}
\end{proposition}
\begin{proof}
We first prove the part (1). To that end, invoking the following transformation formula for the Kummer's function \cite[p.~325, Fomrula 13.2.40]{nist}
\begin{align*}
	U(a; c;z) =z^{1- c} U ( a - c + 1;z - c; z),
\end{align*}
in \eqref{bracket analytic}, we are led to
\begin{align}\label{equiv form}
	(s,\nu k)_k=\sum_{j=0}^k(-s)^j\left(1+\frac{j-1}{s}\right)\cdots \left(1+\frac{1}{s}
	\right) (\nu k)^{k-j} U(k; \nu k+k-s -j- 1;\nu k).
\end{align}
We now need the following asymptotic \cite[p.~144, Equation (10.4.90)]{temme2}
\begin{align*}
	U(a;c;z) &\sim \sqrt{2\pi}\, c^{\,c-\frac{3}{2}} z^{1-c}\left(\frac{c-z}{c}\right)^{a-1} \frac{e^{z-c}}{\Gamma(a)},\quad\quad c\to +\infty.
\end{align*}
Replacing $s$ by $-y,\ y>0$ in \eqref{equiv form} and then employing the above asymptotic with $a=k,\ y=\nu k+k+y -j- 1$ and $z=\nu k$, we obtain
\begin{align}\label{bound for -ve y}
(-y,\nu k)_k\ll y^\delta e^{-y}y^y(\nu k)^{-y},
\end{align}
as $y\to\infty$, where $\delta$ is some fixed real number. Using \eqref{bound for -ve y} and stirling formula \cite[p.~61, Section 3.6]{temme}
\begin{align*}
	n!\sim\sqrt{2\pi n}n^ne^{-n},\qquad n\to\infty,
\end{align*}
in \eqref{exp defn}, we see that
\begin{align*}
\left|\exp_p(-x;z,k)\right|=\left|\sum_{n=0}^{\infty} \frac{(-1)^n}{n!} (z-n, 2\pi pk)_k x^n\right|\ll\sum_{n=0}^\infty\left|\frac{x}{2\pi pk}\right|^n.
\end{align*}
This proves that series is absolutely convergent for $|x|<2\pi pk$.
		
We now prove the second part of the proposition.	

Consider the contour formed by the line segments $[c-iT,c+iT],\ [c+iT,-c_1-iT],\ [-c_1+iT,c_1-iT]$ and $[-c_1-iT,c-iT]$, where $c_1\notin\mathbb{Z}$, $c_1>1$. Note that as the function $(s+z, 2\pi pk)_k$ is entire which can be seen from Proposition \ref{bracket is entire}, the only singularities of the integrand on the right-hand side of \eqref{exp_p mellin eqn} inside contour arise from the poles of $\Gamma(s)$ at $s =-n$ for every $n $ with $0\leq n\leq\lfloor c_1\rfloor$. Moreover, the residue at $s = -n$, which we denote by $R_{-n}$, can be evaluated as
\begin{align}
R_{-n} &= \lim_{s \to -n} (s+n)(s+z, 2\pi pk)_k \Gamma(s) x^{-s}\nonumber \\
    &= (z-n, 2\pi pk)_k  \left\{\lim_{s \to -n} (s+n)\Gamma(s)\right\}x^n \nonumber\\
    &= (z-n, 2\pi pk)_k \frac{(-1)^n}{n!}x^n\nonumber
\end{align}
Hence, by the Cauchy residue theorem, we are led to
\begin{align}\label{cauchy2}
&\frac{1}{2\pi i} \left\{\int_{c-iT}^{c+iT}+\int_{c+iT}^{-c_1+iT}-\int_{-c_1-iT}^{-c_1+iT}+\int_{-c_1-iT}^{c-iT}\right\}\Gamma(s) (s+z, 2\pi pk)_k  x^{-s}ds\nonumber\\
&=\sum_{n=0}^{\lfloor c_1\rfloor} \frac{(-1)^n}{n!} (z-n, 2\pi pk)_k x^n.
\end{align}
By an application of \eqref{stirling} and the growth of $(s,2\pi pk)_k$ in \eqref{bound10} below, it is easy to see that the integrals along the horizontal lines vanish as $T\to\infty$. Therefore, equation \eqref{cauchy2} reduces to
\begin{align}\label{cauchy3}
\frac{1}{2\pi i} \left\{\int_{(c)}-\int_{(-c_1)}\right\}\Gamma(s) (s+z, 2\pi pk)_k  x^{-s}ds&=\sum_{n=0}^{\lfloor c_1\rfloor} \frac{(-1)^n}{n!} (z-n, 2\pi pk)_k x^n.
\end{align}
We first assume that $0<x<\min\{1,2\pi pk\}$ and show that
\begin{align}\label{int zero}
\frac{1}{2\pi i} \int_{(-c_1)}\Gamma(s) (s+z, 2\pi pk)_k  x^{-s}ds\to0
\end{align}
as $c_1\to\infty$. To prove this, we need a bound for $(s, 2\pi pk)_k$ as $s\to\infty$.  We need the following relation between confluent hypergeometric functions \cite[p.~325, Formula 13.2.42]{nist}
\begin{align}
U(a;b;z)&=\frac{\Gamma(1-b)}{\Gamma(a-b+1)}\, M(a;b;z)+\frac{\Gamma(b-1)}{\Gamma(a)}z^{1-b} e^zM(1-a;2-b;-z),\nonumber
\end{align}
and asymptotics \cite[p.~330, Equation (13.8.2); p.~141, Equation (5.11.12)]{nist}
\begin{align}
M(a; b; z) \sim \frac{\Gamma(b)}{\Gamma(b - a)}b^{-a}\qquad\mathrm{and}\qquad \frac{\Gamma(b+a)}{\Gamma(b+c)} \sim b^{a-c} ,\nonumber
\end{align}
as $b\to\infty$ in $|\arg(b)|<\pi$. Here $M(a;b;z)$ is the confluent hypergeometric function \cite[p.~321, Chapter 13]{nist}. Employing these results, we see that
\begin{align}
U(a;b;z)\ll (-b)^{-a}+\Gamma(b-1)z^{1-b}.\nonumber
\end{align}
Using the above bound in \eqref{equiv form}, we obtain
\begin{align}\label{bound10}
(s+z,2\pi pk)_k\ll (-s-z)^{k}\left\{(s+z+1-\nu k)^{-k}+\Gamma(\nu k-s-z- 2)(\nu k)^{s+z-\nu k}\right\}. 
\end{align}
We now make the change of variable $s=-c_1+it$ in the integral in \eqref{int zero} and then employ \eqref{bound10} and \eqref{striling2} to deduce that, as $c_1\to\infty$,
\begin{align}
&\left|\frac{1}{2\pi i} \int_{(-c_1)}\Gamma(s) (s+z, 2\pi pk)_k  x^{-s}ds\right|\nonumber\\
&=\left|\frac{1}{2\pi i} \int_{-\infty}^\infty\Gamma(-c_1+it) (-c_1+it+z, 2\pi pk)_k  x^{c_1-it}dt\right|\nonumber\\
&\ll x^{c_1}\int_{-N}^N\mathcal{O}(1)dt+\left(x^{c_1}+\left(\frac{x}{2\pi pk}\right)^{c_1}\right)\int_{|t|\geq N}\mathcal{O}\left(|t|^{\alpha-\frac{1}{2}}e^{-\frac{\pi}{2}|t|}\right)dt.\nonumber
\end{align}
where $N$ is a sufficiently large positive real number and we used that fact that both the integrals on the right-hand side are finite. Since $0<x<\min\{1,2\pi pk\}$,  it follows that $x^{c_1}\to0$ and $,\left(x/(2\pi pk)\right)^{c_1}\to0$ as $c_1\to\infty$, and hence the integral along the line Re$(s)=-c_1$ tends to zero, proving \eqref{int zero}. 

Now combining \eqref{cauchy3} and \eqref{int zero}, we prove \eqref{exp_p mellin eqn} for $0<x<\min\{1,2\pi pk\}$. An application of \cite[p,~30, Theorem 2.3]{temme} shows that the right-hand side of \eqref{exp_p mellin eqn} is an analytic function of $x$ in the region $|x|<2\pi pk$, and $\exp_p(-x)$ is already an analytic function in this region due to part (1). Hence, by the principle of analytic continuation, we complete the proof of the part (2).

The part (3) follows easily as we can see that the integral on the right-hand side of \eqref{exp_p mellin eqn} is an analytic function of $x$ in $x\in\mathbb{C}\backslash(-\infty,0]$ by an application of  \cite[p,~30, Theorem 2.3]{temme}.
\end{proof}


We are now ready to provide a proof of Theorem \ref{etap at odd}.
\begin{proof}[\textbf{Theorem \textup{\ref{etap at odd}}}][]
We prove the result for $m\in \mathbb{Z}^+$. Similarly, one can easily prove the case $m\in \mathbb{Z}^-$.   

Before proving the transformation formula, it is essential to prove that  the series 
\begin{align}\label{series in}
\sum_{j=1}^{\infty}\frac{p^2+\lambda_j^2}{p\left(p+\frac{1}{\pi }\right)+\lambda_j^2}\Omega_p(\lambda_j x,m)
\end{align}
is absolutely convergent for Re$(x)>0$. To that end, we have to first show that the series defining the function $\Omega_p(x,m)$ in \eqref{omega function} is absolutely convergent. Note that using Proposition \ref{exp_p mellin} with $c>1$, we have
\begin{align}\label{ab omega1}
\left|\Omega_p(x,m)\right|\ll\sum_{k=1}^{\infty}\frac{1}{k^{c+2m+1}} \frac{1}{2\pi} \int_{-\infty}^\infty \left|(c+it+2m+1,2\pi pk)_k \Gamma(c+it)\right|x^{-c}dt.
\end{align}
Equation \eqref{gamma generalized} implies that
\begin{align*}
	\left|(c+it+2m+1,2\pi pk)_k\right|\ll\mathcal{O}(1).
\end{align*}
This along with \eqref{ab omega1} leads to
\begin{align}\label{ab omega}
\left|\Omega_p(x,m)\right|\ll\left\{\sum_{k=1}^{\infty}\frac{1}{k^{c+2m+1}}\right\} \left\{\frac{1}{2\pi} \int_{-\infty}^\infty \left|\Gamma(c+it)\right|x^{-c}dt\right\}.
\end{align}
Observe that the series on the right-hand side of \eqref{ab omega} is just $\zeta(c+2m+1)$ and the integral is absolutely convergent due to \eqref{stirling}. This shows that the series in \eqref{omega function} is absolutely convergent.

Now to show that the series in \eqref{series in} is absolutely convergent, we use the definition of  the function $\Omega_p$ from \eqref{omega function} to see that, for $\Re(s)>1,$
\begin{align}
\sum_{j=1}^{\infty}\frac{p^2+\lambda_j^2}{p\left(p+\frac{1}{\pi }\right)+\lambda_j^2}\Omega_p(\lambda_j x,m)=\sum_{j=1}^{\infty}\frac{p^2+\lambda_j^2}{p\left(p+\frac{1}{\pi }\right)+\lambda_j^2}\sum_{k=1}^{\infty}\frac{\exp_p{(-\lambda_j kx;2m+1,k)}}{k^{2m+1}}.\nonumber
\end{align}
Invoking Proposition \ref{exp_p mellin} with $c>1$ in the above expression, we deduce that
\begin{align}\label{interchange}
&\sum_{j=1}^{\infty}\frac{p^2+\lambda_j^2}{p\left(p+\frac{1}{\pi }\right)+\lambda_j^2}\Omega_p(\lambda_j x,m)\nonumber\\ &=\sum_{j=1}^{\infty}\frac{p^2+\lambda_j^2}{p\left(p+\frac{1}{\pi }\right)+\lambda_j^2}
\sum_{k=1}^{\infty}\frac{1}{k^{2m+1}} \frac{1}{2\pi i} \int_{(c)} (s+2m+1,2\pi pk)_k \Gamma(s)(\lambda_j kx)^{-s}\,ds\nonumber\\
&=\sum_{j=1}^{\infty}\frac{p^2+\lambda_j^2}{p\left(p+\frac{1}{\pi }\right)+\lambda_j^2}
\frac{1}{2\pi i} \int_{(c)}\Gamma(s)\eta_p(s+2m+1)( \lambda_j x)^{-s}\,ds,
\end{align}
where in the last step we interchanged the order of summation and integration, which is justified by absolute convergence \cite[p.~30, Theorem 2.1]{temme}, and then used the definition of $\eta_p(s)$ from \eqref{etap defn}. Equation \eqref{interchange} yields that
\begin{align*}
&\left|\sum_{j=1}^{\infty}\frac{p^2+\lambda_j^2}{p\left(p+\frac{1}{\pi }\right)+\lambda_j^2}\Omega_p(\lambda_j x,m)\right|\ll\left|\sum_{j=1}^{\infty}\frac{p^2+\lambda_j^2}{p\left(p+\frac{1}{\pi }\right)+\lambda_j^2}\frac{1}{\lambda_j^c}\right| \int_{-\infty}^\infty\left|\Gamma(c+it)\eta_p(c+it+2m+1)\right|x^{-c}\,dt.
\end{align*}
Now observing that the series on the right-hand side in the above equation is nothing but $\zeta_p(c)$ and the integral is also convergent due to \eqref{stirling}. This proves the convergence of the series in \eqref{series in}.

We now prove the transformation formula \eqref{etap at odd eqn}.  By performing the interchange of the order of summation and integration on the right-hand side of \eqref{interchange} and then using \eqref{zetap defn}, we are led to
\begin{align}\label{integral_def}
\sum_{j=1}^{\infty}\frac{p^2+\lambda_j^2}{p\left(p+\frac{1}{\pi }\right)+\lambda_j^2}\Omega_p(\lambda_j x,m)
&=\frac{1}{2\pi i} \int_{(c)}\Gamma(s)\zeta_p(s)\eta_p(s+2m+1)( x)^{-s}\,ds\nonumber \\
&=\frac{1}{2\pi i} \int_{(c)} \frac{\eta_p(1-s)\eta_p(s+2m+1)}{2\cos\left(\frac{\pi s}{2}\right)} 
\left( \frac{x}{2\pi} \right)^{-s} \, ds,
\end{align}
which follows upon employing the functional equation \eqref{kosh fe}. We next shift the line of integration from $\Re(s) = c$ to $\Re(s) = -d_1$, where $2m+1 < d_1 < 2m+2$. To this end, we construct a rectangular contour $\mathcal{C}$ with vertices $A(c - iT)$, $B(c + iT)$, $C(-d_1 + iT)$, and $D(-d_1 - iT)$. Note that the integrand
has simple poles at $s \in \{0, -2m\}$ and at $s = -2j+ 1$ for $0 \leq j \leq m+1$, all of which lie inside the contour. By the Cauchy residue theorem, we have
\begin{align}\label{cauchy1}
&\frac{1}{2\pi i} \left\{\int_{c-i\infty}^{c+i\infty} + \int_{c+iT}^{-d_1+iT} + \int_{-d_1+iT}^{-d_1-iT} + \int_{-d_1-iT}^{c-iT} \right\} 
\frac{\eta_p(1-s)\eta_p(s+2m+1)}{2\cos\left(\frac{\pi s}{2}\right)} 
\left( \frac{x}{2\pi} \right)^{-s} \, ds \nonumber\\
&= R_{-2m} + R_0  + \sum_{j=0}^{m+1} R_{-2j+1},
\end{align}
where $R_{z_0}$ denotes the residue of the integrand at $s=z_0$. The residues appearing in \eqref{cauchy} can be evaluated explicitly as
\begin{align*}
R_0 & = -\frac{1}{1+\frac{1}{\pi p}}\frac{1}{2} \eta_p(2m+1), \\
R_{-2m} &= \frac{(-1)^m}{2} \frac{1}{1+\frac{1}{\pi p}} \left( \frac{x}{2\pi} \right)^{2m} \eta_p(2m+1), \\
R_{-2j+1} &= \frac{(-1)^{j+1}}{ \pi}  \left( \frac{x}{2\pi} \right)^{2j-1} \eta_p(2j)\eta_p(2m-2j+2).
\end{align*}
It  is easy to show that the integrals along the horizontal lines in \eqref{cauchy} tend to zero as $T\to\infty$ by invoking \cite[p.~24, Chapter 1, Equation (44)]{koshliakov}
\begin{align*}
\eta_p(\sigma+it)=\mathcal{O}\left(|t|^b\log|t|\right), \quad |t|\to\infty,
\end{align*}
 and the Stirling's formula \eqref{stirling}.

Thus, equations \eqref{integral_def} and \eqref{cauchy1} imply that
\begin{align}
&\sum_{j=1}^{\infty}\frac{p^2+\lambda_j^2}{p\left(p+\frac{1}{\pi }\right)+\lambda_j^2}\Omega_p(\lambda_j x,m) \nonumber\\
&= \frac{(-1)^m}{2} \frac{1}{1+\frac{1}{\pi p}} \left( \frac{x}{2\pi} \right)^{2m} \eta_p(2m+1)+\sum_{j=0}^{m+1}  \frac{(-1)^{j+1}}{ \pi}  \left( \frac{x}{2\pi} \right)^{2j-1} \eta_p(2j)\eta_p(2m-2j+2) \nonumber\\
&\quad-\frac{1}{1+\frac{1}{\pi p}}\frac{1}{2} \eta_p(2m+1) + \frac{1}{2\pi i} \int_{(-d_1)} \frac{\eta_p(1-s)\eta_p(s+2m+1)}{2\cos\left(\frac{\pi s}{2}\right)} 
\left( \frac{x}{2\pi} \right)^{-s} \, ds.\nonumber
\end{align}
Making the change of variable $s \mapsto -s - 2m$ in the integral on the right-hand side of the above equation and rearranging the terms, we are led to
\begin{align}
&\sum_{j=1}^{\infty}\frac{p^2+\lambda_j^2}{p\left(p+\frac{1}{\pi }\right)+\lambda_j^2}\Omega_p(\lambda_j x,m) + \frac{1}{2}\frac{1}{1+\frac{1}{\pi p}} \eta_p(2m+1) \nonumber \\
&= \left( -\frac{x^2}{4\pi^2} \right)^m \left\{ \frac{1}{2\pi i} \int_{(c)} \frac{\eta_p(s+2m+1)\eta_p(1-s)}{2\cos\left(\frac{\pi s}{2}\right)} \left( \frac{2\pi}{x} \right)^{-s} \, ds + \frac{1}{2} \frac{1}{1+\frac{1}{\pi p}} \eta_p(2m+1) \right\} \nonumber \\
&\hspace{3cm} + \sum_{j=0}^{m+1}\frac{(-1)^{j+1}}{ \pi}  \left( \frac{x}{2\pi} \right)^{2j-1} \eta_p(2j)\eta_p(2m-2j+2),\nonumber
\end{align}
where $1<c<2$. Now using \eqref{integral_def} and letting $x = 2\alpha$ such that $\alpha\beta = \pi^2$ and multiplying both sides by $\alpha^{-m}$, we obtain
\begin{align}
\alpha^{-m} &\left\{ \sum_{j=1}^{\infty}\frac{p^2+\lambda_j^2}{p\left(p+\frac{1}{\pi }\right)+\lambda_j^2}\Omega_p(2\lambda_j\alpha,m)+ \frac{1}{2}\frac{1}{1+\frac{1}{\pi p}} \eta_p(2m+1)\right\} \nonumber \\
&= (-\beta)^{-m} \left\{ \sum_{j=1}^{\infty}\frac{p^2+\lambda_j^2}{p\left(p+\frac{1}{\pi }\right)+\lambda_j^2}\Omega_p(2\lambda_j\beta,m)+ \frac{1}{2} \frac{1}{1+\frac{1}{\pi p}} \eta_p(2m+1) \right\} \nonumber \\
&\hspace{3cm} + \frac{1}{\pi^{2m+2}}\sum_{j=0}^{m+1}(-1)^{j+1}\eta_p(2j)\eta_p(2m-2j+2)\alpha^j\beta^{m-j+1}.\nonumber.
\end{align}
In the next step, we use \ref{even eta eqn}  in the finite sum of the above expression and simplify to arrive at
\begin{align*}
\alpha^{-m} &\left\{ \sum_{j=1}^{\infty}\frac{p^2+\lambda_j^2}{p\left(p+\frac{1}{\pi }\right)+\lambda_j^2}\Omega_p(2\lambda_j\alpha,m)+ \frac{1}{2}\frac{1}{1+\frac{1}{\pi p}} \eta_p(2m+1)\right\} \nonumber \\
&= (-\beta)^{-m} \left\{ \sum_{j=1}^{\infty}\frac{p^2+\lambda_j^2}{p\left(p+\frac{1}{\pi }\right)+\lambda_j^2}\Omega_p(2\lambda_j\beta,m)+ \frac{1}{2} \frac{1}{1+\frac{1}{\pi p}} \eta_p(2m+1) \right\} \nonumber \\
&\qquad\qquad +2^{2m}(-1)^m\sum_{j=0}^{m+1}(-1)^j\frac{B_{2j}^{(2,p)}B_{2m-2j+2}^{(2,p)}}{(2j)!(2m-2j+2)!}\alpha^j\beta^{m-j+1}.
\end{align*}
Finally, we replace $j$ by $m-j+1$ in the finite sum on the right hand side to get the required result.
\end{proof}


\begin{proof}[\textbf{Corollary \textup{\ref{ram coro}}}][]
We first show that, for $m\in\mathbb{Z}$,
\begin{align}\label{lambert series special case}
\lim_{p\to\infty}\sum_{j=1}^{\infty}\frac{p^2+\lambda_j^2}{p\left(p+\frac{1}{\pi }\right)+\lambda_j^2}\Omega_p(2\lambda_j\alpha,m)=\sum_{j=1}^\infty\frac{j^{-2m-1}}{e^{2j\alpha}-1}.
\end{align}
Invoking \eqref{exp p infty} in the second step below, we see that
\begin{align}
\lim_{p\to\infty}\sum_{j=1}^{\infty}\frac{p^2+\lambda_j^2}{p\left(p+\frac{1}{\pi }\right)+\lambda_j^2}\Omega_p(2\lambda_j\alpha,m)
&=\lim_{p\to\infty}\sum_{j=1}^{\infty}\frac{p^2+\lambda_j^2}{p\left(p+\frac{1}{\pi }\right)+\lambda_j^2}\sum_{k=1}^\infty\frac{\exp_p(-2k\lambda_j\alpha;2m+1,k)}{k^{2m+1}}\nonumber\\
&=\sum_{j=1}^{\infty}\sum_{k=1}^\infty\frac{\exp(-2kj\alpha)}{k^{2m+1}}\nonumber\\
&=\sum_{k=1}^\infty\frac{k^{-2m-1}}{\exp(2k\alpha)-1}.\nonumber
\end{align}
Now letting $p\to\infty$ on both sides of \eqref{etap at odd eqn} and  using the fact that $\lim_{p\to\infty}\eta_p(s)=\zeta(s)$ together with \eqref{bernoulli special case p infty} and \eqref{lambert series special case}, we complete the proof of the corollary.
\end{proof}

\begin{proof}[\textbf{Corollary \textup{\ref{berndt coro}}}][]
We first show that, for  $m\in\mathbb{Z}$,
\begin{align}\label{exp p0}
\lim_{p\to0}\exp_p(-x\lambda_j,2m+1,k)=(-1)^k\exp(-x(j-1/2)).
\end{align} 
To prove this, note that, as $p\to0$, it is clear from \eqref{sincos eqn} and \eqref{gamma generalized} that $\lambda_j\to j-1/2$ and $(-1)^k$ respectively. Hence, from Proposition \ref{exp_p mellin}, we have 
\begin{align}
\lim_{p\to0}\exp_p(-x\lambda_j,2m+1,k)
&=(-1)^k\exp(-x(j-1/2)).\nonumber
\end{align}
We now need to prove that
\begin{align}\label{lambert series special case p0}
\lim_{p\to0}\sum_{j=1}^{\infty}\frac{p^2+\lambda_j^2}{p\left(p+\frac{1}{\pi }\right)+\lambda_j^2}\Omega_p(2\lambda_j\alpha,m)
=\sum_{j=1}^\infty\frac{(-1)^jj^{-2m-1}}{e^{j\alpha}-e^{-j\alpha}}.
\end{align}
Invoking \eqref{exp p0} in the second step below, we see that
\begin{align}
\lim_{p\to0}\sum_{j=1}^{\infty}\frac{p^2+\lambda_j^2}{p\left(p+\frac{1}{\pi }\right)+\lambda_j^2}\Omega_p(2\lambda_j\alpha,m)
&=\lim_{p\to0}\sum_{j=1}^{\infty}\frac{p^2+\lambda_j^2}{p\left(p+\frac{1}{\pi }\right)+\lambda_j^2}\sum_{k=1}^\infty\frac{\exp_p(-2k\lambda_j\alpha;2m+1,k)}{k^{2m+1}}\nonumber\\
&=\sum_{j=1}^{\infty}\sum_{k=1}^\infty\frac{(-1)^k\exp\left(-2k\left(j-1/2\right)\alpha\right)}{k^{2m+1}}\nonumber\\
&=\sum_{k=1}^\infty\frac{(-1)^k\exp(k\alpha)}{k^{2m+1}}\sum_{j=1}^{\infty}\exp\left(-2kj\alpha\right)\nonumber\\
&=\sum_{k=1}^\infty\frac{(-1)^kk^{-2m-1}}{\exp(k\alpha)-\exp(-k\alpha)}.\nonumber
\end{align}
Now letting $p\to0$ on both sides of \eqref{etap at odd eqn} and  using the fact that \cite[Equation (2.11)]{dixitgupta1}
$$\lim_{p\to0}\eta_p(s)=\left(2^{1-s}-1\right)\zeta(s)$$ together with \eqref{bernoulli special case p0} and \eqref{lambert series special case p0}, we complete the proof of the corollary.
\end{proof}

\begin{proof}[\textbf{Theorem \textup{\ref{t2}}}][]
Replacing $m$ with $-1-m$, where $m\in\mathbb{N}$, in Theorem \ref{etap at odd} yields the result. We observe that the finite sum on the right-hand side of \eqref{etap at odd eqn} vanishes in this case.
\end{proof}


\begin{proof}[\textbf{Theorem \textup{\ref{even integer}}}][]
Letting $\alpha=\beta=\pi$ and replacing $m$ by $2m\ (m\in\mathbb{Z}^+)$ in Theorem \ref{t2}, we obtain the required result.
\end{proof}

\begin{proof}[\textbf{Corollary \textup{\ref{closed form is zero}}}][]
 Equation \eqref{closed form is zero eqn11} is a consequence of letting $p\to0$ in Theorem \ref{even integer} and \eqref{lambert series special case p0}. Equation \eqref{closed form is zero eqn} follows upon letting $p\to\infty$ in Theorem \ref{even integer} and invoking \eqref{lambert series special case} and Corollary \ref{bernoulli special case}.
\end{proof}

\begin{proof}[\textbf{Theorem \textup{\ref{t4}}}][]
On setting $\alpha=\beta=\pi$ and replacing $m$ by $2m+1\ (m\in\mathbb{Z}^+)$ in \eqref{etap at odd eqn}, the required result follows immediately.
\end{proof}

\begin{proof}[\textbf{Corollary \textup{\ref{cauchy}}}][]
Both the results in the corollary follow easily upon taking $p\to\infty$ and $p\to0$ and employing Corollary \ref{bernoulli special case} and equations \eqref{lambert series special case} and \eqref{lambert series special case p0}.
\end{proof}

\begin{proof}[\textbf{Theorem \textup{\ref{closed form}}}][]
Substituting  $m=-1$, 
 $\alpha=\beta=\pi$ into \eqref{etap at odd eqn} and employing the fact that   \cite[Chapter 1, p.~22, Equation (34)]{koshliakov}
 $$\eta_p(0)=-1/2,$$
  yields the required result.
\end{proof}

 \begin{proof}[\textbf{Corollary \textup{\ref{last cor}}}][]
	Let $p\to\infty$ in Theorem \ref{closed form} and use \eqref{lambert series special case} the fact that $\zeta(-1)=-1/12$ to arrive at \eqref{last cor1}. Other equation \eqref{last cor2} follows upon taking $p\to0$ in Theorem \ref{closed form} and using \eqref{lambert series special case p0}.
\end{proof}

\begin{proof}[\textbf{Theorem \textup{\ref{eisent trans}}}][]
 Note that \eqref{t21}  can be re-written as for $m>1$,
\begin{align}
\nonumber 
&\alpha^{m}\biggr\{\sum_{j=1}^{\infty}\frac{p^2+\lambda_j^2}{p\left(p+\frac{1}{\pi }\right)+\lambda_j^2}\Omega_p(2\lambda_j\alpha,-m)+
\frac{1}{2}\frac{1}{1+\frac{1}{\pi p}} \eta_p(-2m+1)\biggr\}\\
&=\left( -\beta \right)^{m} 
\biggr\{
\sum_{j=1}^{\infty}\frac{p^2+\lambda_j^2}{p\left(p+\frac{1}{\pi }\right)+\lambda_j^2}\Omega_p(2\lambda_j\beta,-m)+
\frac{1}{2} \frac{1}{1+\frac{1}{\pi p}}  \eta_p(-2m+1) \biggr\}.\nonumber 
\end{align}
Letting $\alpha = -\pi i z$ and $\beta = \pi i / z$, Im$(z)>0$, in the above equation and then simplifying, we are led to
\begin{align}
\nonumber 
&\sum_{j=1}^{\infty}\frac{p^2+\lambda_j^2}{p\left(p+\frac{1}{\pi }\right)+\lambda_j^2}\Omega_p(-2\lambda_j\pi iz,-m)+
\frac{1}{2}\frac{1}{1+\frac{1}{\pi p}} \eta_p(-2m+1)\\
&=\left( z \right)^{-2m} 
\biggr\{
\sum_{j=1}^{\infty}\frac{p^2+\lambda_j^2}{p\left(p+\frac{1}{\pi }\right)+\lambda_j^2}\Omega_p(2\lambda_j\pi i/z,-m)+
\frac{1}{2} \frac{1}{1+\frac{1}{\pi p}}  \eta_p(-2m+1) \biggr\}.\nonumber 
\end{align}
This implies that
\begin{align}
\nonumber 
&1+\frac{2\left(1+\frac{1}{\pi p} \right)}{\eta_p(1-2m)}\sum_{j=1}^{\infty}\frac{p^2+\lambda_j^2}{p\left(p+\frac{1}{\pi }\right)+\lambda_j^2}\Omega_p(-2\lambda_j\pi iz,-m)\\
&=\left( z \right)^{-2m} 
\biggr\{1+\frac{2\left(1+\frac{1}{\pi p} \right)}{\eta_p(1-2m)}
\sum_{j=1}^{\infty}\frac{p^2+\lambda_j^2}{p\left(p+\frac{1}{\pi }\right)+\lambda_j^2}\Omega_p(2\lambda_j\pi i/z,-m) \biggr\}.\nonumber 
\end{align}
By utilizing the definition of the Eisenstein series $E_k^{(p)}(z)$ as given in \eqref{p-eisen}, we now conclude the proof of the theorem. 
\end{proof}

\begin{proof}[\textbf{Theorem \textup{\ref{quasi trans}}}][]
 If we take $m=-1$ in \eqref{etap at odd eqn} then we have 
 \begin{align}
 	\alpha&\bigg\{\sum_{j=1}^{\infty}\frac{p^2+\lambda_j^2}{p\left(p+\frac{1}{\pi }\right)+\lambda_j^2}\Omega_p(2\lambda_j\alpha,-1)+
 	\frac{1}{2}\frac{1}{1+\frac{1}{\pi p}} \eta_p(-1)\bigg\}\nonumber \\
 	&=-\beta\bigg\{\sum_{j=1}^{\infty}\frac{p^2+\lambda_j^2}{p\left(p+\frac{1}{\pi }\right)+\lambda_j^2}\Omega_p(2\lambda_j\beta,-1)+
 	\frac{1}{2}\frac{1}{1+\frac{1}{\pi p}} \eta_p(-1)
 	\bigg\}-\frac{1}{4}.\nonumber 
 \end{align}
Now setting  $\alpha = -\pi i z$ and $\beta = \pi i / z$ in the above expression and invoking the definition of $E_2^{(p)}(z)$ from \eqref{p-eisen}, we complete the proof.
\end{proof}

\begin{proof}[\textbf{Theorem \textup{\ref{case m=0}}}][]
The argument proceeds in the same spirit as the proof of Theorem \ref{etap at odd}, so we only outline the main ideas. One begins with the infinite sum representing it as the line integral form,
\begin{align}
\sum_{j=1}^{\infty}\frac{p^2+\lambda_j^2}{p\left(p+\frac{1}{\pi }\right)+\lambda_j^2}\Omega_p(1,\lambda_jx)=\frac{1}{2\pi i}\int_{(c)}\frac{\eta_p(1-s)\eta_p(s+1)}{2\cos\left(\frac{\pi s}{2}\right)}
\left( \frac{x}{2\pi} \right)^{-s} \, ds.\nonumber
\end{align}
Now, we shift the integral from $1<\Re(s)=c<2$ to $-2<\Re(s)=d<-1$ then in this rectangular contour, we have a double order pole at $s=0$ and simple poles at $s=-1$ and $s=1$. The rest follows on the similar lines as in the proof of Theorem \ref{etap at odd}.
\end{proof}

 \begin{proof}[\textbf{Corollary \textup{\ref{cauchy10}}}][]
Letting $p\to0$ in Theorem \eqref{case m=0} and using \eqref{lambert series special case p0} and Corollary \ref{bernoulli special case} and the fact that $B_2=1/6$, one obtains the result.
\end{proof}

\subsection{Functional equations for generalized Ramanujan polynomials}
%
%
%
%

\begin{proof}[\textbf{Theorem \textup{\ref{etap functional equation}}}][]

We first prove both the two-term functional equations in \eqref{etap functional equation two term}. Using \eqref{p-ramanujan poly}, we have
\begin{align*}
z^{2k-2}\mathscr{P}_{k}^{(\ell,p)}\left(\frac{1}{z}\right)&=z^{2k-2} \sum_{r=0}^{2k} \frac{B_{r}^{(\ell,p)}B_{2k-r}^{(\ell,p)} }{r!(2k-r)}z^{1-r}\\
    &=\sum_{r=0}^{2k} \frac{B_{r}^{(\ell,p)}B_{2k-r}^{(\ell,p)} }{r!(2k-r)}z^{2k-r-1}.
\end{align*}
Employing the change of variable $r=2k-j$, we see that
\begin{align*}
z^{2k-2}\mathscr{P}_{k}^{(\ell,p)}\left(\frac{1}{z}\right)&=\sum_{j=0}^{2k} \frac{B_{j}^{(\ell,p)} B_{2k-j}^{(\ell,p)} }{j!(2k-j)!}z^{r-1}\\
 &=\mathscr{P}_{k}^{(\ell,p)}(z).
\end{align*}

We next prove three-term functional equations. Let us first consider the case associated with the functions $\mathscr{P}_{k}^{(1,p)}(z)$. To that end,  using the generating function of $p$-Bernoulli numbers $B^{(1,p)}_n$ \cite[Chapter 5]{koshliakov}
\begin{align*}
	\sum_{n=0}^\infty\frac{B^{(1,p)}_n}{n!}t^{n}=\frac{t}{\sigma\left(\frac{t}{2\pi}\right)e^{t}-1},\qquad |t|<2\pi\lambda_1,
\end{align*}
we see that
\begin{align*}
 \frac{zt^2}{\left(\sigma\left(\frac{zt}{2\pi}\right)e^{tz}-1\right)\left(\sigma\left(\frac{t}{2\pi}\right)e^t-1\right)}&=\sum_{n=0}^{\infty} \frac{B_n^{(1,p)}}{n!} (zt)^n \times \sum_{r=0}^{\infty} \frac{B_r^{(1,p)}}{r!} t^r\\ &=\sum_{k=0}^\infty\left(\sum_{r=0}^k\frac{B_r^{(1,p)} z^r}{r!}\frac{B_{k-r}^{(1,p)}}{(k-r)!}\right)t^k.
\end{align*}
Dividing by $z$ both sides, we have
\begin{align*}
 \frac{t^2}{\left(\sigma\left(\frac{zt}{2\pi}\right)e^{zt}-1\right)\left(\sigma\left(\frac{t}{2\pi}\right)e^t-1\right)}=\sum_{k=0}^\infty\left(\sum_{r=0}^k\frac{B_r^{(1,p)}B_{k-r}^{(1,p)}}{r!(k-r)!}z^{r-1}\right)t^k.
\end{align*}
Therefore, $ \mathcal{P}_{k}^{(1,p)}(z) $ is nothing but the coefficient of $t^{2k}$ in  $$H_p(z,t):=\frac{t^2}{\left(\sigma\left(\frac{zt}{2\pi}\right)e^{zt}-1\right)\left(\sigma\left(\frac{t}{2\pi}\right)e^t-1\right)}.$$

Moreover, the functions $\mathscr{P}_{k}^{(\ell,p)}(z-1)$ and $z^{2k-1}\mathscr{P}_{k}^{(\ell,p)}\left(\frac{z-1}{z}\right)$ are coefficients of $t^{2k}$ in the expansions of $H_p(z-1,t)$ and $z^{-2}H_p\left(\frac{z-1}{z},zt\right)$, respectively.

Now consider
\begin{align}\label{before dp}
&H_p(z,t)-H_p(z-1,t)+z^{-2}H_p\left(\frac{z-1}{z},zt\right)\nonumber\\
&=\frac{t^2}{\left(\sigma\left(\frac{t}{2\pi}\right)e^t-1\right)}\left(\frac{1}{\left(\sigma\left(\frac{zt}{2\pi}\right)e^{zt}-1\right)}-\frac{1}{\left(\sigma\left(\frac{(z-1)t}{2\pi}\right)e^{(z-1)t}-1\right)}
\right)+\frac{t^2\left(\sigma\left(\frac{zt}{2\pi}\right)e^{zt}-1\right)^{-1}}{\left(\sigma\left(\frac{(z-1)t}{2\pi}\right)e^{(z-1)t}-1\right)}\nonumber\\
&=\frac{t^2}{\left(\sigma\left(\frac{t}{2\pi}\right)e^t-1\right)}\left(
\frac{\sigma\left(\frac{(z-1)t}{2\pi}\right)e^{(z-1)t}-\sigma\left(\frac{zt}{2\pi}\right)e^{zt}}
{\left(\sigma\left(\frac{zt}{2\pi}\right)e^{zt}-1\right)\left(\sigma\left(\frac{(z-1)t}{2\pi}\right)e^{(z-1)t}-1\right)}
\right)
+\frac{t^2\left(\sigma\left(\frac{zt}{2\pi}\right)e^{zt}-1\right)^{-1}}{\left(\sigma\left(\frac{(z-1)t}{2\pi}\right)e^{(z-1)t}-1\right)}.
\end{align}
Let us define
\begin{align}\label{dp}
	\mathcal{D}_p(x,w):=\frac{2xw(w-x)}
	{(p-x+w)(p-x)(p+w)}.
\end{align}
Using \eqref{sigma}, we have
\begin{align*}
	\sigma(x-w)=\sigma(x)\sigma(-w)+\mathcal{D}_p(x,w).
\end{align*}
Invoking this expression with $w=t/(2\pi)$ and $x=zt/(2\pi)$ in \eqref{before dp} and denoting $\mathcal{D}(zt/(2\pi),t/(2\pi))$ by $\mathcal{D}_p(z,t)$ for simplicity, we are led to
\begin{align*}
&H_p(z,t)-H_p((z-1,t)+z^{-2}H_p\left(\frac{z-1}{z},zt\right)\\
&=\frac{t^2}{\left(\sigma\left(\frac{t}{2\pi}\right)e^t-1\right)}\sigma\left(\frac{zt}{2\pi}\right)e^{zt}\left(
\frac{\sigma\left(-\frac{t}{2\pi}\right)e^{-t}-1}
{\left(\sigma\left(\frac{zt}{2\pi}\right)e^{zt}-1\right)\left(\sigma\left(\frac{(z-1)t}{2\pi}\right)e^{(z-1)t}-1\right)}
\right)\\
&
+
\frac{t^2}{\left(\sigma\left(\frac{(z-1)t}{2\pi}\right)e^{(z-1)t}-1\right)\left(\sigma\left(\frac{zt}{2\pi}\right)e^{zt}-1\right)}+\frac{t^2 e^{(z-1)t}\mathcal{D}_p(z,t)
}{\left(\sigma\left(\frac{t}{2\pi}\right)e^t-1\right)\left(\sigma\left(\frac{zt}{2\pi}\right)e^{zt}-1\right)\left(\sigma\left(\frac{(z-1)t}{2\pi}\right)e^{(z-1)t}-1\right)
}
\\
&=-t^2\sigma\left(\frac{zt}{2\pi}\right)e^{zt}\left(
\frac{\sigma\left(-\frac{t}{2\pi}\right)e^{-t}}
{\left(\sigma\left(\frac{zt}{2\pi}\right)e^{zt}-1\right)\left(\sigma\left(\frac{(z-1)t}{2\pi}\right)e^{(z-1)t}-1\right)}
\right)+\frac{t^2\left(\sigma\left(\frac{zt}{2\pi}\right)e^{zt}-1\right)^{-1}}{\left(\sigma\left(\frac{(z-1)t}{2\pi}\right)e^{(z-1)t}-1\right)}\\
&\hspace{4cm}
+
\frac{t^2e^{(z-1)t}\mathcal{D}_p(z,t)
}{\left(\sigma\left(\frac{t}{2\pi}\right)e^t-1\right)\left(\sigma\left(\frac{zt}{2\pi}\right)e^{zt}-1\right)\left(\sigma\left(\frac{(z-1)t}{2\pi}\right)e^{(z-1)t}-1\right)
}
\\
&=t^2\left(-\frac{\sigma\left(\frac{(z-1)t}{2\pi}\right)e^{(z-1)t}-\mathcal{D}_p(z,t)
e^{(z-1)t}}{\left(\sigma\left(\frac{(z-1)t}{2\pi}\right)e^{(z-1)t}-1\right)\left(\sigma\left(\frac{zt}{2\pi}\right)e^{zt}-1\right)}+\frac{1}{\left(\sigma\left(\frac{(z-1)t}{2\pi}\right)e^{(z-1)t}-1\right)\left(\sigma\left(\frac{zt}{2\pi}\right)e^{zt}-1\right)}
\right)\\
&\hspace{4cm}
+
\frac{t^2e^{(z-1)t}\mathcal{D}_p(z,t)
}{\left(\sigma\left(\frac{t}{2\pi}\right)e^t-1\right)\left(\sigma\left(\frac{zt}{2\pi}\right)e^{zt}-1\right)\left(\sigma\left(\frac{(z-1)t}{2\pi}\right)e^{(z-1)t}-1\right)
}
\\
&=t^2\left(\frac{1-\sigma\left(\frac{(z-1)t}{2\pi}\right)e^{(z-1)t}
}
{\left(\sigma\left(\frac{(z-1)t}{2\pi}\right)e^{(z-1)t}-1\right)\left(\sigma\left(\frac{zt}{2\pi}\right)e^{zt}-1\right)}
\right)+\frac{t^2e^{(z-1)t}\mathcal{D}_p(z,t)}{\left(\sigma\left(\frac{(z-1)t}{2\pi}\right)e^{(z-1)t}-1\right)\left(\sigma\left(\frac{zt}{2\pi}\right)e^{zt}-1\right)}
\\
&\hspace{4cm}
+
\frac{t^2e^{(z-1)t}\mathcal{D}_p(z,t)
}{\left(\sigma\left(\frac{t}{2\pi}\right)e^t-1\right)\left(\sigma\left(\frac{zt}{2\pi}\right)e^{zt}-1\right)\left(\sigma\left(\frac{(z-1)t}{2\pi}\right)e^{(z-1)t}-1\right)
}\\
&=-\frac{t}{z}\frac{tz}{\sigma\left(\frac{zt}{2\pi}\right)e^{zt}-1}+\frac{t^2\sigma\left(\frac{t}{2\pi}\right)e^{zt}\mathcal{D}_p(z,t)
}
{\left(\sigma\left(\frac{t}{2\pi}\right)e^t-1\right)\left(\sigma\left(\frac{zt}{2\pi}\right)e^{zt}-1\right)\left(\sigma\left(\frac{(z-1)t}{2\pi}\right)e^{(z-1)t}-1\right)
}
\\
&=-\sum_{n=0}^\infty\frac{B^{(1,p)}_n}{n!}z^{n-1}t^{n+1}+\frac{
	x(1-x)t^{5} e^{xt}}{4\pi^{3} \left(\left(e^t-1
	\right)\left(p+\frac{t}{2\pi}\right)+\frac{t}{\pi}\right)
	\left(\left(e^{xt}-1
	\right)\left(p+\frac{xt}{2\pi}\right)+\frac{xt}{\pi}\right)}\nonumber\\
&\qquad\qquad\qquad\qquad\qquad\times\frac{1}{\left(\left(e^{(x-1)t}-1
	\right)\left(p+\frac{(x-1)t}{2\pi}\right)+\frac{(x-1)t}{\pi}\right)
},
\end{align*}
where we have used the definitions of $\sigma(t)$ and $\mathcal{D}_p(z,t)$ from \eqref{sigma} and \eqref{dp}, respectively. 

Finally, comparing the coefficients of $t^{2k}$ on both sides and using the definition of $f_p(t)$ from \eqref{ep1}, we complete the proof of \eqref{etap functional equation three term} for $\ell=1$.

We now prove \eqref{etap functional equation three term} in the case of  $\ell=2$ by using the same idea as in the case of $\ell=1$.

Note that, using \eqref{B2m2p}, we are led to
\begin{align*}
 zt^2\sigma_p(zt)\sigma_p(t)&=\sum_{n=0}^{\infty} \frac{B_n^{(2,p)}}{n!} (zt)^n \times \sum_{r=0}^{\infty} \frac{B_r^{(2,p)}}{r!} t^r\\ &=\sum_{k=0}^\infty\left(\sum_{r=0}^k\frac{B_r^{(2,p)} z^r}{r!}\frac{B_{k-r}^{(2,p)}}{(k-r)!}\right)t^k.
\end{align*}
Dividing the both sides by $z$, we have
\begin{align*}
 t^2\sigma_p(zt)\sigma_p(t)=\sum_{k=0}^\infty\left(\sum_{r=0}^k\frac{B_r^{(2,p)}B_{k-r}^{(2,p)}}{r!(k-r)!}z^{r-1}\right)t^k
\end{align*}
We now observe that the function $ \mathscr{P}_{k}^{(2,p)}(z) $ is precisely the coefficient of $t^{2k}$  in  the expansion of $$G_p(z,t):=t^2\sigma_p(tz)\sigma_p(t).$$

Similarly, one can see that the functions $ \mathscr{P}_{k}^{(2,p)}(z-1)$ and $ \mathscr{P}_{k}^{(2,p)}\left(\frac{z-1}{z}\right)$ arise as the coefficients of $t^{2k}$  in $G_p(z-1,t)$ and $z^{-2}G_p\left(\frac{z-1}{z},zt\right)$, respectively. Now, note that
\begin{align}\label{last step 2}
&G_p(z,t)-G_p((z-1,t)+z^{-2}G_p\left(\frac{z-1}{z},zt\right)\nonumber\\
&=t^2\sigma_p(zt)\sigma_p(t)-t^2\sigma_p\left((z-1)t\right)\sigma_p(t)+t^2\sigma_p\left((z-1)t\right)\sigma_p(zt)\nonumber\\
&=t^2\bigl\{\sigma_p(zt)\sigma_p(t)-\sigma_p\left((z-1)t\right)\sigma_p(t)+\sigma_p\left((z-1)t\right)\sigma_p(zt)
\bigl\}\nonumber\\
&=\mathcal{E}^{(2,p)}(z,t),
\end{align}
where $ \mathcal{E}^{(2,p)}(z,t)$ is defined in \eqref{ep2}. Comparing the coefficients of $t^{2k}$ on both sides of \eqref{last step 2}, we complete the proof of \eqref{etap functional equation three term} for $\ell=2$.
\end{proof}

\begin{proof}[\textbf{Corollary \textup{\ref{ram poly as a special case}}}][]
Equation \eqref{ram poly two term} follows easily from \eqref{etap functional equation two term} as $p\to\infty$.

It is straightforward to see that 
\begin{align}
\lim_{p\to\infty}\mathcal{E}^{(1,p)}(z)=0.\nonumber
\end{align}
This shows that \eqref{etap functional equation three term} reduces to \eqref{ram poly three term} in the case of $\ell=1$.

We next show that 
\begin{align}\label{ep1 is zero}
\lim_{p\to\infty}\mathcal{E}^{(2,p)}(z)=0.
\end{align}

Employing the fact that $\lim_{p\to\infty}\sigma_p(t)=\frac{1}{e^t-1}$, we see that
\begin{align*}
&\lim_{p\to \infty}\left(t^2\bigl\{\sigma_p(xt)\sigma_p(t)-\sigma_p\left((x-1)t\right)\sigma_p(t)+\sigma_p\left((x-1)t\right)\sigma_p(xt)
\bigl\}\right)\\
&=\left(\frac{t^2}{(e^{xt}-1)(e^t-1)}-\frac{t^2}{(e^{(x-1)t}-1)(e^t-1)}\right)+\frac{t^2}{(e^{xt}-1)(e^{(x-1)t}-1)}.
\end{align*}
Simplifying the expressions, we arrive at
\begin{align*}
&\lim_{p\to \infty}\left(t^2\bigl\{\sigma_p(xt)\sigma_p(t)-\sigma_p\left((x-1)t\right)\sigma_p(t)+\sigma_p\left((x-1)t\right)\sigma_p(xt)\bigl\}\right)\\
&=-\frac{t^2}{(e^{xt}-1)}\\
&=-\sum_{n=0}^\infty\frac{B_n}{n!}x^{n-1}t^{n+1}.
\end{align*}
Since $B_{2n+1}=0$, the coefficient of $t^{2n}$ on the right-hand side of the above equation is zero, which proves \eqref{ep1 is zero}. This completes the proof.

\end{proof}

\section{Concluding Remarks}
In this paper, we study the special values of the Koshliakov zeta function $\eta_p(s)$ at even and odd integers. More precisely, we derived $p$-analogues of Euler's formula for $\eta_p(2m)$ and of Ramanujan formula for $\eta_p(2m+1)$. Our transformations reduce to known results from the literature as well as to several new results. Our investigation naturally leads to several interesting questions:
\begin{enumerate}
\item Can one prove that the $p$-Bernoulli numbers $B_m^{(2,p)}$ are rational? An affirmative answer to this would imply that the special values $\eta_p(2m),m\geq1,$ are transcendental.

\item Due to Apery's celebrated result \cite{apery, apery2}, we know that $\zeta(3)$ is an irrational number. It is now natural to ask whether a similar result holds for $\eta_p(3)$?

\item It is well known that the classical Bernoulli numbers satisfy many remarkable properties. It would therefore be worthwhile to see whether a parallel theory exists for $B_m^{(2,p)}$.

\item The Koshliakov zeta functions $\zeta_p(s)$ and $\eta_p(s)$ are related through the functional equation \eqref{kosh fe}.  Since the two kinds of $p$-Bernoulli numbers, defined in \eqref{bernoulli p1} and \eqref{B2m2p}, arise from these zeta functions, it is natural to ask whether they are also related by a simple relation.

\item The generalized exponential function arising in the Koshliakov $p$-setting (see \eqref{exp defn}) appears to be of independent interest. A  study of its analytic and structural properties might reveal further connections and applications.

\end{enumerate}

Moreover, our transformation \eqref{etap at odd} yields, as special cases, several results of Berndt, Euler, Ramanujan and others. The limiting case $p\to0$ is particularly interesting, as it leads to explicit closed-form expressions, which in turn yield known transcendence results. In this direction, in a forthcoming paper \cite{gk2}, we consider higher analogues of these results which lead to further new and interesting consequences.

\bigskip
{\bf{Acknowledgements:}} 
The first author acknowledges the support of CSIR for his PhD fellowship. The second author was partially supported by the Grant number ANRF/ ECRG/2024/003222/PMS of Anusandhan National Research Foundation
(ANRF), Govt. of India, and the FIG grants of IIT Roorkee and he acknowledges the support.

\end{document}